\documentclass[a4paper,11pt]{amsart}

\usepackage{fancybox}
\usepackage{amscd}
\usepackage{amsmath}
\usepackage{amssymb}
\usepackage{amsthm}
\usepackage{float}
\usepackage[dvips, pdftex]{graphicx}
\usepackage{tikz}
\usepackage[all, ps, dvips, pdftex]{xy}
\usepackage{color}
\usepackage[hidelinks, hyperfootnotes=false]{hyperref}
\usepackage{array}
\usepackage{amscd}

\DeclareFontFamily{U}{rsfs}{%
\skewchar\font127}
\DeclareFontShape{U}{rsfs}{m}{n}{%
<-6>rsfs5<6-8.5>rsfs7<8.5->rsfs10}{}
\DeclareSymbolFont{rsfs}{U}{rsfs}{m}{n}
\DeclareSymbolFontAlphabet
{\mathrsfs}{rsfs}
\DeclareRobustCommand*\rsfs{%
\@fontswitch\relax\mathrsfs}

\theoremstyle{plain}
\newtheorem{thm}{Theorem}[section]
\newtheorem*{thm*}{Theorem}
\newtheorem{prop}[thm]{Proposition}
\newtheorem{constr}[thm]{Construction}

\newtheorem{defi}[thm]{Definition}
\newtheorem{rmk}[thm]{Remark}

\newtheorem*{cor*}{Corollary}

\newtheorem{prop-defi}[thm]{Proposition-Definition}
\newtheorem{thm-defi}[thm]{Theorem-Definition}
\newtheorem{lem-defi}[thm]{Lemma-Definition}

\newtheorem*{question*}{Question}

\newtheorem{exam}[thm]{Example}

\newtheorem{setup}[thm]{Setup}

\newtheorem{setup-def}[thm]{Setup-Definition}

\newdimen\argwidth
\def\db[#1\db]{
 \setbox0=\hbox{$#1$}\argwidth=\wd0
 \setbox0=\hbox{$\left[\box0\right]$}
  \advance\argwidth by -\wd0
 \left[\kern.3\argwidth\box0 \kern.3\argwidth\right]}

\newcommand{\cC}{\mathcal{C}}

\newcommand{\hH}{\mathcal{H}}

\newcommand{\oO}{\mathcal{O}}

\newcommand{\sS}{\mathcal{S}}

\newcommand{\uU}{\mathcal{U}}

\newcommand{\wW}{\mathcal{W}}
\newcommand{\xX}{\mathcal{X}}

\newcommand{\Ob}{\mathcal{O}b}

\newcommand{\bL}{\mathbb{L}}

\newcommand{\bR}{\mathbb{R}}

\renewcommand{\tilde}{\widetilde}
\renewcommand{\hat}{\widehat}

\newcommand{\dspec}{\mathop{{\textbf{Spec}}}\nolimits}

\newcommand{\lr}{\longrightarrow}

\newcommand{\dR}{\mathbf{R}}

\newcommand{\At}{\mathop{\rm At}\nolimits}

\newcommand{\ch}{\mathop{\rm ch}\nolimits}
\newcommand{\rk}{\mathop{\rm rk}\nolimits}

\newcommand{\Ext}{\mathop{\rm Ext}\nolimits}
\newcommand{\Spec}{\mathop{\rm Spec}\nolimits}

\newcommand{\Coh}{\mathop{\rm Coh}\nolimits}

\newcommand{\im}{\mathop{\rm im}\nolimits}
\newcommand{\re}{\mathop{\rm Re}\nolimits}

\newcommand{\RHOM}{\mathop{\dR {\hH}om}\nolimits}

\newcommand{\tr}{\mathop{\rm tr}\nolimits}

\newcommand{\bC}{\mathbb{C}}
\newcommand{\bZ}{\mathbb{Z}}

\newcommand{\bT}{\mathbb{T}}
\newcommand{\bI}{\mathbb{I}}
\newcommand{\bF}{\mathbb{F}}

\def\vir{\mathrm{\vir}}
\def\loc{\mathrm{\loc}}

\def\unxX{\underline \xX}
\def\unuU{\underline \uU}
\def\Coo{\cC^\infty}
\def\dd{\mathrm{d}}
\def\bT{\mathbb{T}}

\def\ti{\tilde}

\def\sub{\subset}

\def\virt{^{\mathrm{vir}}}

\def\beq{\begin{equation}}
\def\eeq{\end{equation}}

\def\loc{{\mathrm{loc}}}

\makeatletter
 
  \@addtoreset{equation}{section}
\makeatother

\makeatletter
\def\@tocline#1#2#3#4#5#6#7{\relax
  \ifnum #1>\c@tocdepth 
  \else
    \par \addpenalty\@secpenalty\addvspace{#2}%
    \begingroup \hyphenpenalty\@M
    \@ifempty{#4}{%
      \@tempdima\csname r@tocindent\number#1\endcsname\relax
    }{%
      \@tempdima#4\relax
    }%
    \parindent\z@ \leftskip#3\relax \advance\leftskip\@tempdima\relax
    \rightskip\@pnumwidth plus4em \parfillskip-\@pnumwidth
    #5\leavevmode\hskip-\@tempdima
      \ifcase #1
       \or\or \hskip 1em \or \hskip 2em \else \hskip 3em \fi%
      #6\nobreak\relax
    \hfill\hbox to\@pnumwidth{\@tocpagenum{#7}}\par
    \nobreak
    \endgroup
  \fi}
\makeatother

\title[Cosection Localization for D-Manifolds]{Cosection Localization and Vanishing for Virtual Fundamental Classes of D-Manifolds}

\author{Michail Savvas}
\address{Department of Mathematics, The University of Texas at Austin, Austin, TX, 78712, USA}
\email{msavvas@utexas.edu}

\begin{document}

\maketitle

\begin{abstract} 
We establish cosection localization and vanishing results for virtual fundamental classes of derived manifolds, combining the theory of derived differential geometry by Joyce with the theory of cosection localization by Kiem-Li. As an application, we show that the stable pair invariants of hyperk\"{a}hler fourfolds, defined by Cao-Maulik-Toda, are zero.
\end{abstract}

\tableofcontents

\section{Introduction} The study of coherent sheaves and vector bundles on Calabi-Yau manifolds has been a long-standing subject of interest in mathematics and theoretical physics. 

From the point of view of enumerative geometry, the virtual counts of coherent sheaves on Calabi-Yau threefolds are known as Donaldson-Thomas invariants and were first introduced by Thomas \cite{Thomas}. He showed that moduli spaces of stable sheaves admit a perfect obstruction theory \cite{BehFan} and thus a virtual fundamental cycle \cite{BehFan, LiTian}. These moduli spaces enjoy many rich structures and the invariants and their generalizations and refinements have been extensively studied (for a far from exhaustive list of references, cf. \cite{BehFun, JoyceSong, KontSoibel, JoyceSch, JoyceCat, Okou1, KLS, Sav, KiemSavvas}).

Suppose now that $W$ is a smooth, projective complex Calabi-Yau fourfold (meaning that its canonical bundle is trivial, $K_W \cong \oO_W$) and $X$ a proper moduli scheme parameterizing stable sheaves on $W$ with fixed determinant. In this case, the deformations and (higher) obstructions for a sheaf $E$ on $W$ are controlled by the groups
\begin{align*}
    \Ext^1(E,E)_0,\ \Ext^2(E,E)_0,\ \Ext^3(E,E)_0
\end{align*}
and thus the natural candidate complex for a perfect obstruction theory has three terms, making the results of \cite{BehFun, LiTian} and the theory of virtual fundamental cycles not directly applicable.

In order to define Donaldson-Thomas invariants, Cao-Leung \cite{CaoLeungDT4} first suggested a gauge-theoretic approach, which produces invariants in certain cases. Later, in their seminal paper \cite{BorisovJoyce}, Borisov-Joyce generalized this approach to define invariants in great generality, making crucial use of the fact that $X$ is the truncation of a $(-2)$-shifted symplectic derived scheme \cite{PTVV} and using the theory of derived differential geometry \cite{JoyceDMan}. The virtual fundamental classes they define are thus of differential and not algebraic nature and depend on a choice of orientation, which exists by \cite{CaoGrossJoyceOrient}. In \cite{OhThomasI, OhThomasII}, Oh-Thomas develop a definition strictly within the realm of algebraic geometry and prove that their virtual fundamental cycle coincides with the class of Borisov-Joyce.

\medskip

Computing Donaldson-Thomas invariants and invariants of Donaldson-Thomas type for Calabi-Yau fourfolds has attracted significant interest recently, cf. \cite{CaoKoolZero, CaoKoolDT4PT4, CaoKoolMonavariKDT4PT4, CaoKoolMonavariPT4LocalCY, CaoTodaDerivedDT4, CaoTodaGVDescendent, CaoGV0Fano, CaoMaulikToda, CaoMaulikToda2}.

One of the most important methods in handling such invariants in algebraic geometry has been the localization of virtual fundamental cycles by cosections introduced by Kiem-Li \cite{KiemLiCosection}.

The purpose of this paper is to establish cosection localization and vanishing results for the virtual fundamental class appearing in the context of Borisov-Joyce, which will aid in working with these invariants, and at the same time enhance our understanding of virtual classes in the context of derived differential geometry. To do so, we combine the theory of derived manifolds by Joyce \cite{JoyceDMan} with the analytic and topological versions of cosection localization by Kiem-Li \cite{KiemLiSmoothCosection,KiemLiCosection,KiemLiQuantum}. The corresponding algebraic theory for the virtual fundamental cycle constructed by Oh-Thomas has recently been developed by Kiem-Park in \cite{KiemPark}. We expect that our localized virtual fundamental class coincides with the one constructed by Kiem-Park under the appropriate comparison result proved in \cite{OhThomasII}.  

\subsection*{Statement of results} Here we give brief statements of our main results. These fall into three categories depending on the assumptions made: vanishing of the virtual fundamental class in the presence of a surjective weak real cosection, localization by a real cosection, and localization by complex cosections which ``come from geometry".

We have the following vanishing result.

\begin{thm*}[Theorem~\ref{vanishing for surjective real cosection}]
Let $\xX$ be a compact, oriented d-manifold with a continuous family of surjective $\bR$-linear maps $\sigma_x \colon h^1(\bT_{\xX}|_x) \to \bR$ and underlying topological space $X$. Then its virtual fundamental class is zero, i.e. $[\xX]\virt = 0 \in H_{\mathrm{vir.dim.}}(X)$.
\end{thm*}

This statement is essentially already embedded in the literature, following from the definition of the virtual fundamental class of $\xX$ given in \cite{JoyceDMan}.

More generally, when a d-manifold admits a real cosection, we may localize the virtual fundamental class to its degeneracy locus.

\begin{thm*}[Theorem~\ref{localization by real cosection}]
Let $\xX$ be a compact, oriented d-manifold with underlying topological space $X$, equipped with a real cosection $\sigma \colon h^1(\bT_{\xX}|_X) \to \bR_X$ (cf. Definition~\ref{cosection for d-manifold}). Let $X(\sigma)$ be the locus where $\sigma$ is not surjective and $i \colon X(\sigma) \to X$ be the inclusion map. Then the virtual fundamental class localizes to $X(\sigma)$, i.e. there exists a cosection localized virtual fundamental class $[\xX]\virt_{\loc, \sigma} \in H_{\mathrm{vir.dim.}}(X(\sigma))$ satisfying
$$i_* [\xX]\virt_{\loc, \sigma} = [\xX]\virt \in H_{\mathrm{vir.dim.}}(X).$$
\end{thm*}

In \cite{BorisovJoyce}, the virtual fundamental class of a $(-2)$-shifted symplectic derived scheme $(\unxX, \omega_{\unxX})$ is obtained by truncating $\unxX$ in a suitable way to produce an associated d-manifold $\xX_{dm}$ and then taking its virtual fundamental class. In this setting, using the above theorems gives the following vanishing and localization results for cosections on $\unxX$ that are non-degenerate with respect to the symplectic form, since we can show that such cosections induce cosections on $\xX_{dm}$ which are surjective on their non-degeneracy locus.

\begin{thm*}[Theorem~\ref{vanishing for non-degenerate cosection}, Theorem~\ref{cos loc for derived sch}]
Let $(\unxX, \omega_{\unxX})$ be a proper, oriented $(-2)$-shifted symplectic derived scheme with a cosection $\sigma \colon \bT_{\unxX} |_X [1] \to \oO_X$ (cf. Definition~\ref{cosection for derived schemes}). 

If $\sigma$ is non-degenerate on $\unxX$, then the virtual fundamental class vanishes, i.e. $[\xX_{dm}]\virt = 0 \in H_{\mathrm{vir.dim.}}(X)$.

More generally, the virtual fundamental class localizes to the degeneracy locus $X(\sigma)$ of $\sigma$, so that there exists a cosection localized class $[\xX]\virt_{\loc, \sigma} \in H_{\mathrm{vir.dim.}}(X(\sigma))$ satisfying
$$i_* [\xX]\virt_{\loc, \sigma} = [\xX]\virt \in H_{\mathrm{vir.dim.}}(X),$$
where $i \colon X(\sigma) \to X$ denotes the inclusion.
\end{thm*}

As a corollary, we obtain the vanishing of stable pair invariants of hyperk\"{a}hler fourfolds.

\begin{cor*}[Theorem~\ref{vanishing of PT4}, {\cite[Claim~2.19]{CaoMaulikToda}}]
Let $P_n(W, \beta)$ be the moduli space parameterizing stable pairs $I = [\oO_W \xrightarrow{s} F]$ on a hyperk\"{a}hler fourfold $W$ satisfying $[F] = \beta \in H_2(W), \chi(F) = n$. If $\beta \neq 0$ or $n \neq 0$, then the virtual fundamental class is zero, i.e. $[P_n(W, \beta)]\virt = 0$.
\end{cor*}

Finally, when a derived manifold $\xX$ and a cosection $\sigma$ satisfy conditions bringing them closer to being complex analytic or obtained from analytic or algebraic data, we show that can localize the virtual fundamental class by the cosection in a more explicit way of algebraic flavor (Theorem-Definition~\ref{localization by complex cosection}).

\subsection*{Layout of the paper} In \S2, we provide necessary background on d-manifolds, $(-2)$-shifted symplectic schemes and their relationship. \S3 contains material on the topological definition of normal cones and the virtual fundamental class of a d-manifold. In \S4 we establish vanishing results for surjective real cosections and in \S5 we treat the general case of real cosections making use of the results of \S4, especially for their application in derived symplectic algebraic geometry. In \S6 we study complex cosections coming from geometry in more detail. Finally, in \S7 we obtain the vanishing of stable pair invariants of hyperk\"{a}hler fourfolds.

\subsection*{Acknowledgements} The author would like to greatly thank Dominic Joyce, Young-Hoon Kiem and Richard Thomas for helpful conversations related to this work. He would also like to thank the anonymous referee for their comments and suggestions for improvement.

\subsection*{Notation and conventions} 
For topological purposes, we work with Borel-Moore homology and singular (co)homology with integer coefficients. These are equal for compact spaces and afford us the necessary flexibility to work with fundamental classes of (non)compact spaces. We refer the reader to \cite{Bredon, Iversen} for a comprehensive treatment. 

We work with the theory of $\Coo$-schemes and derived manifolds developed by Joyce \cite{JoyceDMan}. In analogy to usual smooth manifolds, our $\Coo$-schemes will be separated, second countable and locally fair, meaning that their underlying topological space is Hausdorff, second countable and they are locally modelled by the spectrum of a finitely generated $\Coo$-ring. By definition, such $\Coo$-schemes can be the underlying scheme of a d-manifold. Unless otherwise stated, the underlying topological spaces of $\Coo$-schemes and d-manifolds are assumed to be Euclidean Neighbourhood Retracts (ENR) for simplicity. This is automatically the case for the complex analytic spaces or algebraic schemes appearing throughout the paper.

All algebraic and derived schemes are defined over the field of complex numbers $\bC$. Classical and derived algebraic schemes and complex analytic spaces are separated and of finite type. The reader can consult \cite{Toen} for more general background on derived algebraic geometry and shifted symplectic structures or \cite{BorisovJoyce} for the necessary material for our applications.

$\unxX$ will typically denote a derived scheme and $\omega_{\unxX}$ a $(-2)$-shifted symplectic form on $\unxX$. $\xX$ will typically denote a d-manifold and $X_{\Coo}$ its underlying $\Coo$-scheme. $X$ will always denote the underlying topological space of $\unxX$ or $\xX$ and sometimes, by abuse of notation, the classical truncation of $\unxX$ as well. In that case, $X^{an}$ denotes the complex analytic space associated to classical truncation $X$ of $\unxX$. These distinctions will be explicitly stated or should be clear from context throughout the paper. In general, we typically use the notation $X_{\Coo}$ when we need to use the existence of quasicoherent sheaves on a $\Coo$-scheme, whereas we only consider homology groups of the underlying topological spaces.

\section{Background on D-Manifolds and $(-2)$-Shifted Symplectic Derived Schemes} \label{Background section}

In this section, we collect some requisite background and introduce terminology that will be used in the rest of the paper.

\subsection{$\cC^\infty$-schemes and d-manifolds}  The main references for the theory of $\Coo$-schemes and d-manifolds are the books \cite{JoyceCoo} and \cite{JoyceDMan} by Joyce. Here we give a brief informal account of their structure.

Roughly speaking, under the conventions of this paper, a $\Coo$-scheme is a locally ringed space that is locally isomorphic to the (real) spectrum of the $\Coo$-ring $\Coo(\bR^n) / I$, where $I$ is an ideal of smooth functions on $\bR^n$. There is a theory of (quasi)coherent sheaves and vector bundles on $\Coo$-schemes, which resembles that of sheaves on algebraic schemes.

D-manifolds form a theory of derived differential geometry being analogous to quasi-smooth derived or dg-schemes. Their local structure is given as follows.

\begin{defi} \emph{\cite[Example~1.4.4]{JoyceDMan}}
A \emph{principal d-manifold} $\xX = \sS_{Y,E,s}$ is determined by the data of a smooth manifold $Y$, a smooth vector bundle $E \to Y$ and $s \in \cC^\infty (E)$ as follows: Its underlying $\Coo$-scheme $X_{\Coo}$ is the affine scheme obtained as the spectrum of the $\Coo$-ring $\cC^\infty(Y)/I$, where $I=(s)$ is the ideal generated by the section $s$. The higher data are given by the morphism $E^\vee |_{X_{\Coo}} \xrightarrow{s^\vee} I / I^2$ of $\oO_{X_{\Coo}}$-modules, so that the underlying topological space $X$ is the zero locus of the section $s$. The tangent complex (or virtual tangent bundle) of $\xX$ is the two-term complex
$$\bT_{\xX} = [ T_Y |_{X_{\Coo}} \xrightarrow{\dd s} E|_{X_{\Coo}}]$$
where for any choice of connection $\nabla$ on $E^\vee$, $\dd s$ is the restriction to $X_{\Coo}$ of $\nabla s \colon T_Y \to E$.
\end{defi}

In general, every d-manifold $\xX$ is locally equivalent to a principal d-manifold. When $\xX$ is compact (i.e. its underlying topological space is compact), Joyce proves a Whitney embedding theorem, which in particular implies that this is true globally. Since we will be fundamentally interested in compact d-manifolds in this paper, this allows us in practice to consider the case of principal d-manifolds, which are a considerable simplification.

\begin{thm-defi} \emph{\cite[Theorems~4.29, 4.34]{JoyceDMan}} \label{presentation as a d-manifold}
Let $\xX$ be a compact d-manifold. Then $\xX \simeq \sS_{Y,E,s}$ where $Y$ is an open subset of some Euclidean space $\bR^N$. We refer to this equivalence as a presentation of $\xX$ as a principal d-manifold. The virtual dimension of $\xX$ is then equal to $N - \rk E$.
\end{thm-defi}

In analogy with the case of quasi-smooth derived schemes or schemes with a perfect obstruction theory \cite{BehFan}, we give the following definition.

\begin{defi}
The \emph{obstruction sheaf} of $\xX$ is the coherent sheaf $\Ob_{\xX} := h^1(\mathbb{T}_\xX) \in \Coh(X_{\cC^\infty})$, where $\mathbb{T}_\xX$ is the tangent complex of $\xX$.
\end{defi}

Finally we mention that there is a notion of orientation for a d-manifold $\xX$, which in the case of a principal d-manifold $\sS_{Y,E,s}$ roughly amounts to an orientation of the determinant line bundle $\det \bT_\xX$. As in the case of manifolds, compact, oriented d-manifolds admit a virtual fundamental class.

\begin{thm} \emph{\cite[Section~13]{JoyceDMan}}
Let $\xX$ be a compact, oriented d-manifold with underlying topological space $X$. Then there exists a virtual fundamental class $[\xX]\virt \in H_{\mathrm{vir.dim.}}(X)$. Moreover, $[\xX]\virt$ only depends on the bordism class of $\xX$.
\end{thm}

\subsection{$(-2)$-shifted symplectic derived schemes} In this subsection, we establish some terminology and describe local structure results that will be used in the rest of the paper. For an introduction to derived algebraic geometry, we refer the reader to \cite{Toen}. Shifted symplectic structures were introduced in the seminal paper \cite{PTVV}.

In a nutshell, a derived scheme $\unxX$ is locally modelled by the (derived) spectrum $\dspec A$ where $A$ is a commutative differential (negatively) graded $\bC$-algebra. A $(-2)$-shifted symplectic structure on $\unxX$ is given by a 2-form
\begin{equation*}
    \omega_{\unxX} \colon \bT_{\unxX} \wedge \bT_{\unxX} \lr \oO_{\unxX}[-2]
\end{equation*}
which is non-degenerate and is equipped with extra data that make it closed.

A $(-2)$-shifted symplectic scheme is a pair $(\unxX, \omega_{\unxX})$ consisting of a derived scheme $\unxX$ and a symplectic structure on it. There is also an appropriate notion of orientation on a $(-2)$-shifted symplectic scheme (cf. \cite{BorisovJoyce}). Given the above, we introduce the following terminology.

\begin{defi}
The obstruction sheaf of $\unxX$ is the sheaf $\Ob_{\unxX} = h^1(\mathbb{T}_{\unxX} |_{X})$. The obstruction space of $\underline{\xX}$ at a point $x \in \underline{\xX}$ is the vector space $\Ob(\underline{\xX}, x) = h^1(\mathbb{T}_{\underline{\xX}}|_x)$.
\end{defi}

In particular, the symplectic form $\omega_{\underline{\xX}}$ induces a family of non-degenerate quadratic forms $q_x \colon \Ob(\unxX, x) \to \bC$.

We conclude this subsection with a local description of $(-2)$-shifted symplectic schemes that will be essential later on. In \cite{JoyceSch}, the local structure of $(\unxX, \omega_{\unxX})$ around every point $x \in X$ is described by charts in Darboux form as follows.

\begin{thm-defi} \emph{\cite[Example~5.16]{JoyceSch}} \label{Darboux chart}
Let $x \in \unxX$ be a point of a $(-2)$-shifted symplectic derived scheme $(\unxX, \omega_{\unxX})$. Then there exists a Zariski open neighbourhood $\underline{U} = \dspec A \to \unxX$ around $x$ such that:
\begin{enumerate}
    \item $(A, \delta)$ is a commutative differential graded algebra, whose degree zero part $A^0$ is a smooth $\bC$-algebra of dimension $m$, with a set of \'{e}tale coordinates $\{ x_i \}_{i=1}^m$. 
    \item $A$ is freely generated over $A^0$ by variables $\{ y_j \}_{j=1}^n$ and $\{ z_i \}_{i=1}^m$ in degrees $-1$ and $-2$ respectively.
    \item Let $\omega_A$ be the pullback of $\omega_{\unxX}$ to $\dspec A$. There exist invertible elements $q_1, ..., q_n \in A^0$ such that
    \begin{equation*}
        \omega_A = \sum_{i=1}^m \dd z_i \dd x_i + \sum_{j=1}^n \dd (q_j y_j) \dd y_j .
    \end{equation*}
    \item The differential $\delta$ is determined by the equations
    \begin{align*}
        \delta x_i = 0, \quad \delta y_j = s_j, \quad \delta z_i = \sum_{j=1}^n y_j \left( 2q_j \frac{\partial s_j}{\partial x_i} + s_j \frac{\partial q_j}{\partial x_i} \right)
    \end{align*}
    where the elements $s_j \in A^0$ satisfy
    \begin{equation*}
        q_1 s_1^2 + ... + q_n s_n^2 = 0.
    \end{equation*}
\end{enumerate}

We let $V = \Spec A^0$, $E$ the trivial vector bundle of rank $n$ whose dual has basis given by the variables $y_j$ and $F$ the trivial vector bundle whose dual has basis given by the variables $z_i$. We refer to all of the above data as a \emph{derived algebraic Darboux chart} for $\unxX$ at $x$.

At the classical level, the truncation $U = h^0(\underline{U}) \sub X$ is equipped with the following data:
\begin{enumerate}
    \item A smooth affine scheme $V$ of dimension $m$.
    \item A trivial vector bundle $E$ on $V$ of rank $n$ with a non-degenerate quadratic form $q$ which is given by
    \begin{align*}
        q_u(y_1, ..., y_n) = q_1(u) y_1^2 + ... + q_n(u) y_n^2
    \end{align*}
    on each fiber of $E$ over $u \in U$.
    \item A $q$-isotropic section $s \in \Gamma(E)$ whose scheme-theoretic zero locus is $U \sub V$.
    \item A three-term perfect complex of amplitude $[0,2]$
    \begin{align*}
        G := \bT_{\underline{U}}|_U = [T_V|_U \xrightarrow{ds} E|_U \xrightarrow{t} F|_U].
    \end{align*}
    such that $q_u$ gives a non-degenerate quadratic form on each obstruction space $h^1(G|_u) = \Ob(\unxX, u)$.
\end{enumerate}

We refer to the data $(V,E,G,q,s)$ as an \emph{algebraic Darboux chart} for $X$ at $x$.

By working in the complex analytic category, since the functions $q_j$ are non-zero at $x$, by possibly shrinking $V$ around $x$, they admit square roots $q_j = r_j^2$ and after re-parametrizing $y_j \mapsto \frac{y_j}{r_j}$ we may assume that $q_j(u) = 1$. We then refer to the data $(V,E,G,q,s)$ as an \emph{analytic Darboux chart} for $X$ at $x$.
\end{thm-defi}

Different Darboux charts glue and thus are compatible through appropriate data of homotopical nature. These compatibilities are important, however we will not need them explicitly. Knowing that they exist and that local constructions play well with them (which is ensured by the results we quote) will be sufficient for our purposes.

\subsection{From $(-2)$-shifted symplectic schemes to d-manifolds} \label{subsection 2.3} Consider a $(-2)$-shifted symplectic derived scheme $(\underline{\xX}, \omega_{\unxX})$. Let $X = h^0(\unxX)$ be its classical truncation and, by abuse of notation, we write $X$ for the underlying topological space as well.

In this subsection, we briefly recall the main results of \cite{BorisovJoyce}, where the authors produce a (non-uniquely determined) d-manifold $\xX_{dm}$ associated to $(\underline{\xX}, \omega_{\unxX})$ and thus a virtual fundamental class $[\xX_{dm}]\virt \in H_*(X)$.

In order to construct a d-manifold starting with a $(-2)$-shifted symplectic derived scheme, it is necessary to take appropriate (real smooth) truncations of Darboux charts.

\begin{defi} \emph{\cite[Definition~3.6]{BorisovJoyce}} \label{q-adapted definition} Let $(V,E,G,q,s)$ be data of an analytic Darboux chart for $X$ at the point $x$. We say that a real subbundle $E^- \sub E$ is \emph{$q$-adapted} if for all $u \in U$ the following conditions are met:
\begin{enumerate}
    \item $\im(ds|_u) \cap E^- |_u = \{ 0 \} \sub E|_u$.
    \item $t|_u (E^-|_u) = t|_u(E|_u) \sub F|_u$.
    \item The image of the linear map 
    $$\Pi_u \colon \ker(t|_u) \cap E^-|_u \lr \ker(t|_u) \lr h^1(G|_u) = \Ob(\unxX,u)$$
    is a real maximal negative definite vector subspace of $\Ob(\unxX,u)$ with respect to the quadratic form $\re q_u$.
\end{enumerate}
\end{defi}

Here is a (very) brief outline of the steps involved in the construction of the d-manifold $\xX_{dm}$.

\begin{constr} \label{construction of d-man from derived scheme} \text{ } 
\begin{enumerate}
    \item $(\unxX, \omega_{\unxX})$ admits a cover by derived algebraic Darboux charts, which are compatible in a suitable homotopy-theoretic sense.
    \item The classical truncation of these charts gives a cover of $X$ by analytic Darboux charts.
    \item For each such chart $(V,E,G,q,s)$ at a point $x \in X$, up to possible shrinking of $V$ around $x$, there exists a $q$-adapted $E^- \sub E$. Fix any such choice and let $E^+ = E / E^-$ and $s^+ = s \text{ mod } E^-$.
    \item Using the compatibility data between the derived algebraic Darboux charts, the principal d-manifolds $\sS_{V,E^+,s^+}$ glue to a d-manifold $\xX_{dm}$ with underlying topological space $X^{an}$.
\end{enumerate}
\end{constr}

Moreover, orientations of $(\unxX, \omega_{\unxX})$ are in bijection with orientations of $\xX_{dm}$. Thus a proper, oriented $(\unxX, \omega_{\unxX})$ gives rise to a compact, oriented d-manifold $\xX_{dm}$ and a virtual fundamental class $[\xX_{dm}]\virt \in H_*(X)$.

A priori, depending on the choices of $q$-adapted subbundles $E^- \sub E$, there are many possibilities for the d-manifold $\xX_{dm}$ produced by this process. However, it is shown in \cite{BorisovJoyce} that all such d-manifolds have the same bordism class and therefore must have the same virtual fundamental class. We record this in the following proposition.

\begin{prop} \emph{\cite[Corollary~3.19]{BorisovJoyce}} \label{virtual class is canonical}
The virtual fundamental class $[\xX_{dm}]\virt \in H_*(X)$ associated to a proper, oriented $(-2)$-shifted symplectic derived scheme $(\unxX, \omega_{\unxX})$ is independent of any choices involved in the construction of the d-manifold $\xX_{dm}$ and thus canonical.
\end{prop}

\section{Normal Cones and Virtual Fundamental Classes} \label{Normal Cones Section}

In this short section, we first briefly recall a topological version of the algebraic intrinsic normal cone \cite{BehFan} and its basic properties, following \cite[Section~3]{Siebert}. We then use the normal cone to give an alternative definition of the virtual fundamental class of a compact, oriented d-manifold $\xX$, which is equivalent to the one given in \cite{JoyceDMan}.

\subsection{Topological and homological normal cone} Let $Y$ be an oriented, topological manifold of dimension $N$ and $E$ a topological vector bundle on $Y$ with projection map $q \colon E \to Y$. Let furthermore $s$ be a continuous section of $E$ with zero locus $X = Z(s) \sub Y$.

For any $\ell > 0$, let $p_\ell \colon Y \times \bR^\ell \to Y$ be the projection and define a section $s_\ell \in \Gamma \left( Y \times (\bR^\ell \backslash \lbrace 0 \rbrace), p_\ell^* E \right)$ by the formula
\begin{align*}
    s_\ell (y, v) = |v|^{-1} \cdot s(y).
\end{align*}

Let $\overline{\Gamma}_{s_\ell} \sub p_\ell^* E$ denote the closure of the graph of $s_\ell$ in $p_\ell^* E$.

\begin{thm-defi} \cite[Definition~3.1, Proposition~3.2]{Siebert}\label{topological intrinsic normal cone definition}
The \emph{topological normal cone} $C(s) \sub E$ associated to $s$ is defined as
\begin{align} \label{formula of top intr normal cone}
    C(s) = \overline{\Gamma}_{s_\ell} \cap (E \times \lbrace 0 \rbrace).
\end{align}
This is independent of $\ell$ and satisfies:
\begin{enumerate}
    \item $C(s) = \lim_{t \to \infty} \Gamma_{t \cdot s}$, where $C(s)$ and $\Gamma_{t \cdot s}$ are considered as closed subsets of $E$.
    \item $C(s)$ lies over $X$ and $q(C(s)) = X$.
\end{enumerate}
\end{thm-defi}

The long exact sequence in homology gives 
\begin{align*}
    H_{N+\ell}(C(s)) \lr H_{N+\ell}(\overline{\Gamma}_{s_\ell}) \lr H_{N+\ell}({\Gamma}_{s_\ell}) \lr H_{N+\ell-1}(C(s)) 
\end{align*} 
and therefore for large enough $\ell$ the two outer homology groups vanish and the homology class $$(s_\ell)_* [Y \times (\bR^\ell \backslash \lbrace 0 \rbrace)] \in H_{N+\ell}({\Gamma}_{s_\ell})$$ 
gives rise to a fundamental class 
\begin{align} \label{loc 3.2}
    [\overline{\Gamma}_{s_\ell}] \in H_{N+\ell}(\overline{\Gamma}_{s_\ell}).
\end{align}

Here we have chosen an orientation of $\bR^\ell$. This choice will not affect the homological normal cone, defined below.

Let $\delta_0 \in H_{\lbrace 0 \rbrace}^\ell (\bR^\ell)$ be Poincar\'{e} dual to $\lbrace 0 \rbrace \sub \bR^\ell$ and $q_\ell \colon p_\ell^* E \to \bR^\ell$ be the projection.

\begin{thm-defi} \cite[Definition~3.3, Proposition~3.4]{Siebert} \label{homological intrinsic normal cone} The \emph{homological normal cone} associated to $s$ is defined as
\begin{align} \label{definition of homological normal cone}
    [C(s)] = [\overline{\Gamma}_{s_\ell}] \cap q_\ell^* \delta_0 \in H_N ( C(s) ).
\end{align}
It is independent of $\ell$ and satisfies $[C(s)] = [\Gamma_{t \cdot s}] \in H_N ( E )$ for any $t \neq 0$.
\end{thm-defi}

\begin{rmk}
The homological normal cone should be regarded as the fundamental class of $C(s)$. This is motivated by the fact that formula \eqref{formula of top intr normal cone} can obviously be equivalently written as $C(s) = \overline{\Gamma}_{s_\ell} \cap q_\ell^{-1}(0)$.
\end{rmk}

Finally, we remark that the topological normal cone is compatible with the intrinsic normal cone used in complex algebraic geometry. Suppose that $Y$ is a complex manifold, $E$ is a holomorphic vector bundle and $s$ a holomorphic section. Let $I$ be the ideal sheaf of $X$ in $Y$. The normal cone of $X$ associated to $s$ is then defined as
$$ C_{X/Y} = \Spec_{\oO_X} \left( \oplus_{n \geq 0} I^n/I^{n+1} \right)$$
and is a closed subcone of the vector bundle $E$. Here $\Spec$ is used to stand for $\mathrm{MaxSpec}$, so that as a topological space the points of $C_{X/Y}$ correspond to maximal ideals.

\begin{prop} \cite[Proposition~3.5]{Siebert} \label{compatibility with algebraic normal cone}
$C_{X/Y} = C(s) \sub E$ and $[C_{X/Y}] = [C(s)] \in H_N(E)$.
\end{prop}

\subsection{Virtual fundamental class} \label{virtual fundamental cycle subsection} Let now $\xX$ be a compact, oriented d-manifold with underlying topological space $X$ and virtual dimension $n$.

By Theorem-Definition~\ref{presentation as a d-manifold}, $\xX$ is equivalent to a principal d-manifold, meaning that there exists an open subset $Y \sub \bR^N$, a smooth vector bundle $E$ on $Y$ of rank $r$ and a smooth section $s$ of $E$ whose zero locus is $X$. 

Moreover, we may assume that $Y$ retracts onto $X$ and 
$$f_* \colon H_n (X) \to H_n (Y)$$ 
is an isomorphism, where $f \colon X \to Y$ is the inclusion map and $n=N-r$.

The line bundle $L = \Lambda^r E \otimes \Lambda^N T^* Y$ is oriented. Since $\Lambda^N T^* Y$ admits an orientation by the definition of $Y$, this is equivalent to a choice of orientation of the bundle $E$. We may therefore further assume that $E$ is oriented. This choice of orientation does not affect the virtual fundamental class below. 

By a small perturbation of the section $s$ of $E$, we get a section $s'$ of $E$ intersecting the zero section $Y$ transversely. Thus its zero locus $X'$ is a compact, oriented smooth manifold. Write $f' \colon X' \to Y$ for the inclusion.

\begin{defi} \cite{JoyceDMan}
The J-virtual fundamental class of $\xX$ is defined as
\begin{align} \label{Joyce definition of virtual cycle}
    [\xX]_J\virt = (f_*)^{-1} \circ f'_* [X'] \in H_{n}(X).
\end{align}
\end{defi}

It is shown in \cite{JoyceDMan} that the J-virtual fundamental class is independent of all choices involved and only depends on the bordism class of $\xX$. 

\begin{rmk}
The assumption that $Y$ retracts onto $X$ and the pushforward $f_*$ induces an isomorphism on homology uses the fact that $X$ is assumed to be a Euclidean Neighbourhood Retract (ENR). This is however not strictly necessary in order to define $[\xX]_J\virt$, as is explained for example in \cite[Subsection~2.6.5]{BorisovJoyce}.
\end{rmk}

Using the homological normal cone, we can also define a virtual fundamental class as follows.

\begin{defi} \label{BF definition}
The BF-virtual fundamental class of $\xX$ is defined as 
\begin{align} \label{BF definition of virtual cycle}
    [\xX]_{BF}\virt = 0_{E|_X}^! [C(s)] \in H_{n}(X).
\end{align}
\end{defi}

The next proposition shows that the two definitions are equivalent. In particular, the BF-virtual fundamental class is well defined and independent of any choices made above.

\begin{prop}
$[\xX]_J\virt = [\xX]_{BF}\virt$.
\end{prop}

\begin{proof}
By \eqref{Joyce definition of virtual cycle} and \eqref{BF definition of virtual cycle}, we need to show
\begin{align} \label{loc 3.4}
   f'_* [X'] = f_* 0_{E|_X}^! [C(s)] \in H_n(Y).
\end{align}
By the Cartesian diagram
\begin{align*}
    \xymatrix{
    E|_X \ar[d] \ar[r]^g & E \ar[d] \\
    X \ar[r]_f & Y
    }
\end{align*}
we have (cf. for example \cite[Equation~(2.23)]{KiemLiQuantum}) 
\begin{align} \label{loc 3.5}
    f_* \circ 0_{E|_X}^! = 0_E^! \circ g_* .
\end{align}

By Theorem-Definition~\ref{homological intrinsic normal cone}, $g_* [C(s)] = [\Gamma_s] \in H_N(E)$ and for a small perturbation $s'$ of $s$ we have $[\Gamma_s] = [\Gamma_{s'}] \in H_N (E)$. Therefore
\begin{align} \label{loc 3.6}
    0_E^! \ g_* [C(s)] = 0_E^! [\Gamma_{s'}] = f_*'[X'] \in H_n(Y)
\end{align}
since $s'$ is taken to be transverse to the zero section of $E$.

Combining \eqref{loc 3.4}, \eqref{loc 3.5} and \eqref{loc 3.6} completes the proof.
\end{proof}

From now on, we write $[\xX]\virt$ for the virtual fundamental class of the d-manifold $\xX$.

\begin{rmk}
The construction of $[\xX]\virt$ does not strictly require that $\xX$ is compact, but rather the existence of a presentation $\xX \simeq \sS_{Y,E,s}$ or more generally of a closed embedding of $\xX$ as a d-manifold into a smooth manifold. The compactness of $\xX$ implies both of these conditions and simplifies our arguments, so we have opted to assume it here and subsequently in the paper to facilitate the exposition.
\end{rmk}

\begin{rmk}
Suppose that $Y$ is a complex manifold and $E$ and $s$ are holomorphic. Since the homological normal cone $[C(s)]$ coincides with the normal cone $[C_{X/Y}]$ by Proposition~\ref{compatibility with algebraic normal cone}, it is clear by Definition~\ref{BF definition} that the virtual fundamental class $[\xX]\virt$ coincides with the definition of the virtual fundamental cycle in \cite{BehFan}.
\end{rmk}

\section{Vanishing for Surjective Real Cosections} \label{Surjective Cosections Section}

Let $\xX$ be a d-manifold whose truncation is the $\cC^\infty$-scheme $X_{\cC^\infty}$. We typically denote the underlying topological space by $X$.

To simplify notation, we will write $\bR_X$ to denote both the trivial line bundle on $X_{\cC^\infty}$ as an $\oO_{X_{\Coo}}$-module and the trivial real line bundle on the topological space $X$. 

In this section, we introduce the appropriate notions of cosections that will be considered throughout the paper and prove vanishing theorems for the virtual fundamental classes of d-manifolds and $(-2)$-shifted symplectic schemes in the presence of a surjective or non-degenerate cosection respectively, which constitutes the simplest possible scenario.

Besides setting up notation, we treat the case of surjective cosections separately from the general case, which is the subject of Section~\ref{Real Cosections Loc Section}, for two reasons. The first is technical and has to do with the fact that to prove vanishing we don't need to work with normal cones nor employ homological considerations for intersections or the topology of the vanishing locus of a cosection. This simplifies the arguments and also allows us to prove a vanishing result for not just surjective real cosections of d-manifolds, but for the relaxed version of surjective weak real cosections. The second reason is that for the purposes of applying the theory to $(-2)$-shifted symplectic derived schemes, the surjective case (Theorem~\ref{vanishing for non-degenerate cosection}) plays an important role as the first step in the proof of the general case (Theorem~\ref{localization for derived scheme non-degenerate cos}). We therefore elect to treat it separately in order to clarify the argument for the convenience of the reader.

\subsection{Localization by surjective real cosections of d-manifolds} Let $\xX$ be a d-manifold with underlying $\cC^\infty$-scheme $X_{\cC^\infty}$. We give the following definition. 

\begin{defi} \label{cosection for d-manifold}
A \emph{real cosection} is a morphism $\sigma \colon \Ob_{\xX} \to \bR_X$ of coherent sheaves on $X_{\cC^\infty}$.

A \emph{weak real cosection} is a collection of linear maps $\sigma_x \colon \Ob_{\xX}|_x \to \bR$ for $x \in X$ that vary continuously, i.e. for any open subset $U \sub X$ and $\Coo$-vector bundle $E_U$ on $U$ with a surjection $p_U \colon E_U \to \Ob_{\xX}|_U$ of sheaves on $U_{\cC^\infty}$ the linear maps $\sigma_x \circ p_U|_x \colon E_U|_x \to \bR$ define a continuous map of vector bundles $E_U \to \bR_U$.

A (weak) real cosection is surjective if the corresponding morphism(s) in its definition is (are) surjective.
\end{defi}

\begin{rmk}
Clearly a real cosection induces data of a weak real cosection.
\end{rmk}

We then have the following vanishing theorem.

\begin{thm} \label{vanishing for surjective real cosection}
Let $\xX$ be a compact, oriented d-manifold of virtual dimension $n$ with a surjective weak real cosection. Then 
$$[\xX]\virt = 0 \in H_n (X),$$
i.e. the virtual fundamental class of $\xX$ is zero.
\end{thm}

\begin{proof}
By Theorem-Definition~\ref{presentation as a d-manifold}, $\xX$ is equivalent to a principal d-manifold. Therefore, we have an open subset $Y \sub \bR^N$, a smooth, oriented vector bundle $E$ on $Y$ of rank $r$ and a smooth section $s$ of $E$ with $X$ being its zero locus, which define an embedding of $\cC^\infty$-schemes $X_{\cC^\infty} \to Y$.

The tangent complex of $\xX$ is given by
\begin{align*}
    \mathbb{T}_\xX = [T_Y|_{X_{\cC^\infty}} \xrightarrow{\dd s} E|_{X_{\cC^\infty}}]
\end{align*}
and hence we have a surjection $p \colon E|_{X_{\cC^\infty}} \to \Ob_{\xX}$. By the definition of weak cosection, the maps $\sigma_x \circ p|_x \colon E|_x \to \bR$ for $x \in X$ vary continously and thus induce a continuous surjection of vector bundles
\begin{align} \label{loc 4.1}
   \sigma \circ p \colon E|_X \to \bR_X. 
\end{align}

Keeping the same notation around \eqref{Joyce definition of virtual cycle}, with $f \colon X \to Y$ denoting the inclusion of $X$ inside $Y$, we have by definition that 
$$f_\ast [\xX]\virt = e(E) \cap [Y] \in H_n(Y)$$
where $e(E)$ denotes the Euler class of $E$. 

After possibly shrinking $Y$ around $X$, since $X$ is compact, we may assume that the surjection \eqref{loc 4.1} extends to a continuous surjection $E \to \bR_Y$. Since $E$ has a trivial quotient, $e(E) = 0 \in H^{r}(Y)$ and we obtain 
$$[\xX]\virt = (f_\ast)^{-1} \left( e(E) \cap [Y] \right) = 0 \in H_n(X).$$ \end{proof}

\subsection{Vanishing for $(-2)$-shifted symplectic derived schemes} We can give the following general definition of a cosection for any derived scheme $\unxX$. 

\begin{defi} \label{cosection for derived schemes}
A \emph{cosection} is a morphism $\sigma \colon \mathbb{T}_{\unxX} |_X[1]  \to \oO_{X}$ in $D^b(\Coh X)$. We say that $\sigma$ is surjective if the morphism $h^0(\sigma) \colon \Ob_{\unxX} \to \oO_X$ is surjective, or equivalently if the collection of $\bC$-linear maps $\sigma_x \colon \Ob(\unxX, x) \to \bC$ for every $x \in X$ are surjective. 

Now suppose that $\omega_{\unxX}$ is a $(-2)$-shifted symplectic form on $\unxX$. We say that $\sigma$ is non-degenerate if the induced quadratic form $q_x^\vee$ on $\Ob(\unxX, x)^\vee$ is non-degenerate on the image of $\sigma_x^\vee$ for all $x \in X$. In particular, $\sigma$ must be surjective.
\end{defi}

\begin{rmk}
When $\unxX$ is quasi-smooth so that $\bT_{\unxX}|_X$ is a perfect complex of amplitude $[0,1]$, the above definition of a cosection coincides with the standard definition for schemes with perfect obstruction theory, as $h^1(\bT_{\unxX}|_X)$ is the obstruction sheaf of the perfect obstruction theory $\bL_{\unxX}|_X \to \bL_X$ on $X$, in the sense of Behrend-Fantechi \cite{BehFan}.
\end{rmk}

According to our discussion in Subsection~\ref{subsection 2.3}, a $(-2)$-shifted symplectic derived scheme $(\unxX, \omega_{\unxX})$ gives rise to an associated d-manifold $\xX_{dm}$. When $(\unxX, \omega_{\unxX})$ is moreover proper and oriented, there is a well defined virtual fundamental class $[\xX_{dm}]\virt \in H_*(X)$ in the homology of the underlying topological space $X$.

We then have the following theorem, saying that cosection localization applies to the setting of non-degenerate cosections, implying the vanishing of the virtual fundamental class.

\begin{thm} \label{vanishing for non-degenerate cosection}
Suppose that a proper, oriented $(-2)$-shifted symplectic scheme $(\unxX, \omega_{\unxX})$ admits a non-degenerate cosection $\sigma$. These data give rise to a d-manifold $\xX_{dm}$ with an induced surjective real cosection and hence a vanishing virtual fundamental class, i.e. $[\xX_{dm}]\virt = 0 \in H_{\mathrm{vir.dim.}}(X)$.
\end{thm}

\begin{proof} We use the terminology introduced in Subsection~\ref{subsection 2.3}.

By Proposition~\ref{virtual class is canonical}, we are free to make any choice of $q$-adapted $E^- \sub E$ for an analytic Darboux chart $(V,E,G,q,s)$ around a point $x \in X$ in Costruction~\ref{construction of d-man from derived scheme}. 

Since $\bT_{\unxX}|_U$ is isomorphic to the complex $G$, $\sigma$ is determined by a morphism $E \to \oO_U$ and for every $u \in U$ we have the factorization
\begin{align*}
    \sigma_u \colon \Ob (\unxX, u) \sub E|_u / \im (ds|_u) \lr \oO_X|_u
\end{align*}
and hence the projection
\begin{align*}
    E|_u \lr E|_u / \im (ds|_u) \lr \oO_X |_u
\end{align*}
is surjective. Therefore we get a surjective morphism $E|_U \to \oO_U$, which up to possible shrinking of $V$ around $x$, extends to a surjective morphism $\sigma_E \colon E \to \oO_V$.

The dual $\sigma_E^\vee \colon \oO_V \to E^\vee$ restricts to $\sigma_x^\vee$ at $x$, as the differentials of the complex $G$ vanish at $x$ by construction. Thus, since $\sigma$ is non-degenerate, we may assume, up to further shrinking of $V$ around $x$, that the image of $\sigma_E^\vee$ is non-degenerate with respect to the quadratic form $q^\vee$ on $E^\vee$.

Taking the orthogonal complement of the image and dualizing, we get a $q$-orthogonal splitting $E = \oO_V \oplus \ker (\sigma_E)$, such that $q$ is non-degenerate on each summand. After re-parametrizing analytically, $q$ is equal to $$q_v(z) = y_1^2 + y_2^2 + ... + y_n^2$$ on each fiber over $v \in V$, where $y_1$ is the coordinate on the fiber of the first trivial summand $\oO_V$ and $y_2, ..., y_n$ are the fiber coordinates of the summand $\ker(\sigma_E)$.

We now take $E^-$ to be the direct sum of the imaginary parts of each individual summand. $E^-$ is easily seen to be $q$-adapted around $x \in X$ (it is obviously $q$-adapted at $x$ and being $q$-adapted is an open condition by \cite{BorisovJoyce}).

Then, for $u \in X$ close to $x$, the subspace $\im \Pi_u \sub \Ob(\unxX, u)$ must project onto $\bC$ through $\sigma_u$ non-trivially since it is maximal negative definite with respect to $\re q_u$ as $E^-$ is $q$-adapted (see Definition~\ref{q-adapted definition}). But, by the choice of $E^-$, $\im \Pi_u$ consists only of imaginary vectors and therefore its projection to $\bC$ is contained in $i\bR \sub \bC$, so we must have $\sigma_u (\im \Pi_u) = i \bR$. This implies that $\sigma_u$ induces a surjection
\begin{equation*}
    \sigma_u^+ \colon \Ob(\unxX, u) / \im \Pi_u \lr \bR.
\end{equation*}

For every such $u \in X$, the conditions in Definition~\ref{q-adapted definition} imply the existence of an exact sequence
\begin{equation*}
    T_V|_u \xrightarrow{ds^+} E^+ |_u \lr \Ob(\unxX, u) / \im \Pi_u \lr 0.
\end{equation*}

Since $\xX_{dm}$ is locally equivalent to the principal d-manifold $\sS_{V, E^+, s^+}$, the maps $\sigma_u^+ \colon \Ob_{\xX_{dm}}|_u \to \bR$ vary smoothly and, in particular, give a surjective real cosection for $\xX_{dm}$. Thus, by Theorem~\ref{vanishing for surjective real cosection}, we get the desired vanishing $[\xX_{dm}]\virt = 0 \in H_* (X)$.
\end{proof}

\section{Localization by Real Cosections} \label{Real Cosections Loc Section}

Keeping the notation of the previous section, let $\xX$ be a d-manifold whose truncation is the $\cC^\infty$-scheme $X_{\cC^\infty}$ and whose underlying topological space is $X$. Let $\sigma \colon \Ob_{\xX} \to \bR_X$ be a real cosection on $\xX$ (cf. Definition~\ref{cosection for d-manifold}). 

In this section, we use the intersection pairing in Borel-Moore homology in order to localize the virtual fundamental class $[\xX]\virt$ to a class $[\xX]_{\loc, \sigma}\virt \in H_*(X(\sigma))$, where $X(\sigma)$ is the degeneracy locus of $\sigma$, i.e. the locus of points $x \in X$ for which ${\sigma}|_x$ is equal to zero (or more generally possibly undefined). To do this, we first establish the appropriate necessary cone reduction criterion and then adapt the construction first introduced in \cite{KiemLiSmoothCosection} and \cite[Appendix]{KiemLiCosection}. We finally obtain a corresponding cosection localization formalism for $(-2)$-shifted symplectic derived schemes, combining the results of this section with the proof of Theorem~\ref{vanishing for non-degenerate cosection}.

\subsection{Cone reduction by real cosections} In this subsection, we verify that the arguments of \cite[Section~4]{KiemLiCosection} go through in our context. 
\medskip

Let $Y$ be a smooth manifold of dimension $N$, $E$ a smooth vector bundle on $Y$ and $s$ a smooth section of $E$ with zero locus $X = Z(s) \sub Y$. Locally, let $s_1, ..., s_r \in \cC^\infty(Y)$ be the components of the section $s$ and $y_1, ..., y_N$ coordinates on $Y$.

Now consider the principal d-manifold $\xX = \sS_{Y,E,s}$ determined by the above data. Write $p \colon E|_{X_{\Coo}} \to \Ob_{\xX}$ for the projection morphism. We have the following cone reduction statement (cf. \cite[Corollary~4.5]{KiemLiCosection}).

\begin{prop} \label{support of cone}
Let $\bar{\sigma} \colon \Ob_{\xX} \to \bR_X$ be a morphism of coherent sheaves on $X_{\cC^\infty}$. Let $X(\sigma)$ be the locus where $\bar{\sigma} \circ p$ is zero and $U(\sigma) = X - X(\sigma)$ its complement, where $\bar{\sigma} \circ p$ is surjective. Then $C(s)$ is contained entirely in the kernel bundle $E(\sigma) = \ker ( \bar{\sigma} \circ p ) = E|_{X(\sigma)} \cup \ker (\bar{\sigma} \circ p)|_{U(\sigma)}$. 
\end{prop}

\begin{proof}
Let $c = (x, e) \in C(s)$. We need to show that if $x \notin X(\sigma)$, then $c \in \ker (\bar{\sigma} \circ p)|_{U(\sigma)}$. As $X(\sigma)$ is closed in $X$, by shrinking around $x$ we may assume that $\sigma := \bar{\sigma} \circ p \colon E|_{X_{\Coo}} \to \bR_X$ is surjective and extends to a smooth surjective map $E \to \bR_Y$ on $Y$ and that we have fixed a connection $\nabla$ giving the tangent complex of $\xX$, $\dd s \colon T_Y|_{X_{\Coo}} \to E|_{X_{\Coo}}$. 

We have by definition $\sigma|_{X_{\Coo}} \circ \dd s|_{X_{\Coo}} = 0$. This implies that for every variable $y_j$ of $Y$
\begin{equation} \label{loc 4.3''}
    \sum_{i=1}^r \sigma_i \frac{\partial s_i}{\partial y_j} = O(s)
\end{equation}
where $O(s)$ stands for a linear combination of the functions $s_1, ..., s_r$ with coefficients smooth functions, and $\sigma_i$ are the components of $\sigma$.

Fix $\ell = 1$ and $c_1 = (x, 0, e) \in C(s) \sub p_1^* E$. By the definition of $C(s)$, we have a sequence $(y_k, v_k) \in Y \times \bR$ such that, as $k \to \infty$, $y_k \to x, v_k \to 0$ and for $i=1, ..., r$
\begin{equation} \label{loc 4.3'''}
    \lim_{k \to \infty}  \frac{s_i(y_k)}{|v_k|} = e_i.
\end{equation}

Take any continuous path $f(t) = (\gamma(t), \zeta(t)) \colon [0, \epsilon) \to Y \times \bR$ such that $f(0) = (x, 0)$, $f$ is differentiable on $(0, \epsilon)$, $\zeta(t)$ is positive and strictly increasing on $(0, \epsilon)$ and there exists a sequence $t_k \to 0$ in $(0, \epsilon)$ such that $f(t_k) = (y_k, v_k)$.

By re-parametrizing we may assume that $\zeta(t) = t$ and $t_k = v_k = |v_k|$.

\eqref{loc 4.3''} gives then for $t \neq 0$
\begin{equation*}
    \sum_{i=1}^r \sigma_i(\gamma(t)) \frac{\partial s_i}{\partial y_j} (\gamma(t)) \gamma_j'(t) = O(s(\gamma (t)))
\end{equation*}
and hence by summing over $j=1,...,N$
\begin{equation*}
    \sum_{i=1}^r \sigma_i(\gamma(t)) \frac{\dd}{\dd t} s_i(\gamma(t)) = O(s(\gamma(t))).
\end{equation*}
It is immediate that
\begin{equation*}
    \frac{\dd}{\dd t} \sum_{i=1}^r \sigma_i(\gamma(t)) s_i(\gamma(t)) = \sum_{i=1}^r a_i(t) s_i(\gamma(t))
\end{equation*}
for some continuous functions $a_i(t)$, smooth on $(0, \epsilon)$. We integrate to get
\begin{equation*}
    \sum_{i=1}^r \sigma_i(\gamma(t)) s_i(\gamma(t)) = \int_0^t \sum_{i=1}^r a_i(\lambda) s_i(\gamma(\lambda)) \dd \lambda.
\end{equation*}
Evaluating at $t_k$ and dividing by $t_k$ gives
\begin{equation*}
    \sum_{i=1}^r \sigma_i(y_k) \frac{s_i(y_k)}{t_k} = \frac{1}{t_k} \int_0^{t_k} \sum_{i=1}^r a_i(\lambda) s_i(\gamma(\lambda)) \dd \lambda
\end{equation*}
so letting $k \to \infty$ it follows by \eqref{loc 4.3'''} that
\begin{equation*}
    \sigma(e) = \sum_{i=1}^r \sigma_i(x) e_i = \sum_{i=1}^r a_i(0) s_i(x) = 0
\end{equation*}
which is what we want.
\end{proof}

\subsection{Cosection localization for d-manifolds} 
Let us first recall very briefly the intersection pairing in Borel-Moore homology. For more details, we refer the reader to \cite{Iversen} or \cite{Bredon}.

Let $W$ be a smooth oriented manifold of dimension $d$. Then, for any two closed subsets $Z_1, Z_2$, there exists a (bilinear) intersection pairing
\begin{align} \label{intersection pairing}
\cap \colon H_i(Z_1) \times H_j(Z_2) \lr H_{i+j-d}(Z_1 \cap Z_2).
\end{align}
It has a concrete geometric interpretation when $W$ is real analytic and $Z_1, Z_2$ are closed real analytic subsets of $W$, given for example in \cite[Section~1]{Ginzburg} (cf. also \cite[Appendix~A.9]{IonelParker} for a short treatment of the pairing).
\medskip

Let now $\xX$ be a compact, oriented d-manifold of virtual dimension $n$. By Theorem-Definition~\ref{presentation as a d-manifold}, $\xX$ is equivalent to a principal d-manifold $\sS_{Y,E,s}$. We fix a choice of such presentation
\begin{equation} \label{principal model 5.2}
    \xX \simeq \sS_{Y,E,s}.
\end{equation}

Thus we have an embedding $X_{\cC^\infty} \to Y$, given by the data of an open $Y \sub \bR^N$, a smooth, oriented vector bundle $E$ on $Y$ of rank $r$, a smooth section $s$ of $E$ with $X$ being its zero locus and a surjection $p \colon E|_{X_{\Coo}} \to \Ob_{\xX}$. We have $n = N-r$.

Write $\sigma'$ for the composition ${\sigma} \circ p \colon E|_{X_{\Coo}} \to \bR_X$ and $U(\sigma) = X - X(\sigma)$, the locus where $\sigma'$ is surjective.
\medskip

Up to possible shrinking of $Y$, we may take $\sigma_Y' \colon E \to \bR_Y$ to be a smooth extension of $\sigma'$ to the ambient space $Y$ with degeneracy locus $Y(\sigma_Y')$ and complement $U(\sigma_Y') = Y - Y(\sigma_Y')$. 

We now assume that the degeneracy locus $X(\sigma)$ is a Euclidean Neighbourhood Retract (ENR). In particular, this implies that since $X(\sigma)$ is a closed subset of $X$, there is an open neighbourhood $V$ of $X(\sigma)$ in $X$ that retracts onto $X(\sigma)$ and such that the natural map 
\begin{align} \label{loc 5.5}
H_*(X(\sigma)) \lr H_*(V)
\end{align}
induced by the inclusion morphism is an isomorphism.

$\sigma'|_{X-V}$ is then surjective on the closed subset $X-V$ of $Y$ and extends smoothly to the open neighbourhood $U(\sigma_Y')$. Following \cite{KiemLiSmoothCosection}, after choosing a metric on $E$, we can pick a smooth section $\check{\sigma}_Y \in \cC^\infty(E)$ that is a small perturbation of the zero section of $E$ and moreover satisfies that the composition 
\begin{align} \label{def of sigma check}
\sigma'|_{X-V} \circ \check{\sigma}_Y|_{X-V}
\end{align} 
is nonzero everywhere on $X-V$. We denote the image of $\check{\sigma}_Y$ inside $E$ by $\xi$.

We now have homology classes 
\begin{align*}
[C(s)] \in H_N(C(s)) \text{  and  } [\xi] \in H_N(\xi).
\end{align*} 
The latter exists due to the fact that $\xi$ is diffeomorphic to $Y$, which is an oriented manifold. Clearly, being a section of a vector bundle, $\xi$ is a closed subset of the smooth manifold $E$.

By Proposition~\ref{support of cone}, $C(s)$ is contained in $E(\sigma) = E|_{X(\sigma)} \cup \ker(\sigma')|_{U(\sigma)}$, so by the definition of the section $\xi$, we deduce that $C(s) \cap \xi$ maps diffeomorphically under the bundle projection $q \colon E \to Y$ to a closed subspace of $V$.

With these in place, we may apply the homological intersection pairing \eqref{intersection pairing} to the closed subsets $C(s)$ and $\xi$ of the smooth, oriented manifold $E$ of dimension $N+r$ and the homology classes $[C(s)], [\xi]$ respectively to obtain a class
\begin{align} \label{loc 5.3}
[C(s)] \cap [\xi] \in H_{N-r}(C(s) \cap \xi) = H_n(C(s) \cap \xi).
\end{align}

Hence, pushing forward via $q$, the class \eqref{loc 5.3} gives rise to a homology class in $H_{n} (V)$ and hence, applying the isomorphism \eqref{loc 5.5}, to a class
\begin{align} \label{loc 5.7}
[\xX]_{\loc, \sigma}\virt := [C(s)] \cap [\xi] \in H_{n}(X(\sigma)).
\end{align}

By standard homological arguments, one may easily verify that $[\xX]_{\loc, \sigma}\virt$ is independent from the choice of extension $\sigma_Y'$ and section $\check{\sigma}_Y$. By \cite{KiemLiSmoothCosection}, it satisfies that 
$$
i_\ast [\xX]_{\loc, \sigma}\virt = [\xX]\virt \in H_n(X)
$$ 
where $i \colon X(\sigma) \to X$ is the natural inclusion map.

We moreover have the following proposition. 

\begin{prop} \label{prop 5.2}
$[\xX]_{\loc,\sigma}\virt$ is independent from the choice of presentation $\xX \simeq \sS_{Y,E,s}$ and thus well defined.
\end{prop}

\begin{proof}
Suppose that $\xX \simeq \sS_{Y_1,E_1,s_1} \simeq \sS_{Y_2,E_2,s_2}$. By \cite[Theorems~1.4.8, 4.9, 4.34]{JoyceDMan}, up to possibly shrinking $Y_1$ around $X$, we may reduce to the case of an equivalence in the following standard form: we have a submersion $f \colon Y_1 \to Y_2$ and a morphism of vector bundles $\hat{f} \colon E_1 \to f^* E_2$ such that $\hat{f} \circ s_1 = f^* s_2$. The morphism of tangent complexes
\begin{align} \label{loc 5.9}
    \xymatrix{
    T_{Y_1} |_X \ar[r]^-{\dd s_1} \ar[d]_-{\dd f} & E_1|_X \ar[d]^-{\hat{f}} \\
    f^* T_{Y_2} |_X \ar[r]^-{\dd (f^* s_2)} & f^\ast E_2 |_X
    }
\end{align}
is a quasi-isomorphism on fibers over $x \in X$ and thus we must have that $E_1|_X \xrightarrow{\hat{f}} f^* E_2|_X$ is surjective, so after further possible shrinking of $Y_1$, we may assume that $\hat{f}$ is surjective. 

This situation can be locally described as follows: Let $N_1 = \dim Y_1$, $N_2 = \dim Y_2 = N_1 -p$, $y_1, ..., y_{N_1}$ be local coordinates for $Y_1$, $y_1, ..., y_{N_2}$ local coordinates for $Y_2$ and $E_1 = f^* E_2 \oplus \bR_{Y_1}^p$ be a splitting for $\hat{f}$. Then, with respect to this splitting, we have $s_1(y_1, ..., y_{N_1}) = \left( s_2(y_1, ..., y_{N_2}), y_{N_2 + 1}, ..., y_{N_1} \right)$.

Let $i_1 \colon X \to Y_1$, $i_2 \colon X \to Y_2$ be the inclusion maps. They satisfy $i_2 = f \circ i_1$ and fit into a commutative diagram
\begin{align*}
\xymatrix{
X \ar[dr]_-{i_2} \ar[r]^-{i_1} & Y_1 \ar[d]^-{f} \\
 & Y_2.
}
\end{align*}

It follows that the projection
$$\hat{f} \colon E_1 \lr f^* E_2,$$ 
restricted to $X$ via pullback by the inclusion $i_1$, induces a projection map $\hat{f}|_X \colon C(s_1) \to C(s_2)$ which makes $C(s_1)$ into a vector bundle over $C(s_2)$.

By restricting to $X$ and using the inclusion $C(s_2) \sub E_2|_X$ (cf. Theorem-Definition~\ref{topological intrinsic normal cone definition}), we also have a commutative diagram
\begin{align*}
\xymatrix{
 & f^* E_2 \ar[d] \\
C(s_2) \ar[r]_-{j_2} \ar[ur]^-{j_1} & E_2.
}
\end{align*}

Now, let $\sigma_{Y_2}' \colon E_2 \to \bR_{Y_2}$ be a smooth extension of the composition $\sigma \circ p_2 \colon E_2|_{X_{\cC^\infty}} \to \bR_X$. Then, by \eqref{loc 5.9}, $\sigma_{Y_1}' := f^* \sigma_{Y_2} \circ \hat{f}$ is a smooth extension of the corresponding composition $\sigma \circ p_1 \colon E_1|_{X_{\cC^\infty}} \to \bR_X$. 

Pick a splitting $\check{f} \colon f^* E_2 \to E_1$ of $\hat{f}$. Moreover, let $\check{\sigma}_{Y_2}$ be a smooth section of $E_2$ as in \eqref{def of sigma check}. Then, up to rescaling $\check{f}$, the composition $\check{\sigma}_{Y_1} := \check{f} \circ f^* \check{\sigma}_{Y_2}$ also satisfies the same properties in regards to the cosection $\sigma_{Y_1}'$.

Using the splitting, write $E_1 = f^* E_2 \oplus E_3$ so that $\hat{f}$ is the projection onto the first factor and $\check{f}$ is the inclusion of the first factor.

We then have by definition 
$$C(s_1) = C(s_2) \oplus E_3|_X$$
where $C(s_2)$ is a closed subspace of $f^* E_2$ via the embedding $j_1$ and 
$$\check{\sigma}_{Y_1} = f^* \check{\sigma}_{Y_2} \oplus 0.$$

The intersections $C(s_1) \cap \xi_1, C(s_2) \cap \xi_2$ are closed subspaces of $V$, using Lemma~\ref{support of cone} and the construction of $\xi_1$ and $\xi_2$, and we see that they must be equal.

Under this identification it is routine to verify that the properties of the intersection pairing imply that
\begin{align*}
[C(s_1)] \cap [\xi_1] = [C(s_2)] \cap [\xi_2] \in H_n(X(\sigma))
\end{align*}
where we have suppressed the pushforward to $V$ and the isomorphism \eqref{loc 5.5} on homology. \end{proof}

If $X(\sigma)$ is not an ENR or more generally does not admit a neighbourhood $V$ in $X$ satisfying the above properties, we may still perform the construction with the following modifications, following \cite[Subsection~2.6.5]{BorisovJoyce}: Instead of $V$, we consider a sequence of open neighbourhoods 
$$V_1 \supset V_2 \supset V_3 \supset ...$$ 
of $X(\sigma)$ in $X$ satisfying
\begin{align} \label{loc 5.10}
\bigcap_{i=1}^\infty V_i = X(\sigma).
\end{align}

The above construction gives a sequence of homology classes for $i=1,2, \ldots$
$$[\xX]_{\loc, \sigma, i}\virt \in H_n(V_i)$$
compatible under restriction from $V_i$ to $V_{i+1}$ and hence an element
\begin{align} \label{loc 5.11}
[\xX]_{\loc, \sigma}\virt \in \varprojlim_{i \geq 1} H_n(V_i).
\end{align}

Borel-Moore homology satisfies Milnor's continuity axiom (cf. for example \cite[Appendix A]{IonelParker}) and hence 
\begin{align} \label{loc 5.12}
\varprojlim_{i \geq 1} H_n(V_i) = H_n(\cap_{i=1}^\infty V_i).
\end{align}
Combining \eqref{loc 5.10}, \eqref{loc 5.11} and \eqref{loc 5.12}, we again obtain a localized homology class
\begin{align} \label{loc 5.13}
[\xX]_{\loc, \sigma}\virt \in H_n(X(\sigma)).
\end{align}

Summarizing our discussion, we may now give the following definition.

\begin{thm-defi} \label{localization by real cosection}
Let $\xX$ be a compact, oriented d-manifold of virtual dimension $n$ and $\sigma \colon \Ob_{\xX} \to \bR_X$ a real cosection. Then the cosection localized virtual fundamental class $[\xX]_{\loc, \sigma}\virt$ of $\xX$ is defined by \eqref{loc  5.7} (and more generally by \eqref{loc 5.13} when $X(\sigma)$ is not a Euclidean Neighbourhood Retract) and satisfies
\begin{align*}
    i_* [\xX]_{\loc,\sigma}\virt = [\xX]\virt \in H_n(X)
\end{align*}
where $i \colon X(\sigma) \to X$ denotes the inclusion.
\end{thm-defi}

\begin{rmk} \label{meromorphic cosection remark}
In order to define the cosection localized virtual fundamental class, we do not need a globally defined cosection $\sigma$ on $\xX$, but only a surjective cosection $\sigma_{\uU(\sigma)} \colon \Ob_{\uU(\sigma)} \to \bR_{U(\sigma)}$ defined on the open sub-d-manifold $\uU(\sigma)$ of $\xX$ whose underlying topological space is an open subset $U(\sigma)$ inside $X$. We will use this observation in Subsection~\ref{cos loc for derived sch}.
\end{rmk}

\begin{rmk}
The above definition generalizes the statement of Theorem~\ref{vanishing for surjective real cosection} when one only considers cosections of d-manifolds (and not weak cosections).
\end{rmk}

\begin{exam}
Let $Y \sub \bR^N$ be the open unit disk, $E$ an oriented, rank $N$ vector bundle on $Y$, $s$ the zero section of $E$ and $X \simeq \sS_{Y,E,s}$ so that the underlying topological space is $X = Y$. Consider a cosection $\sigma \colon E \to \bR_Y$ whose zero locus $Y(\sigma)$ is the origin, viewed as a smooth manifold.

Clearly $C(s) = Y \sub E$ is the zero section. Let $V$ be a disk of sufficiently small radius $\epsilon$ centered at the origin $0$ . By trivializing $E \cong \bR^N \times \bR^N$ around $0$, we write $y=(y_1, ..., y_N)$ for the coordinates of $Y$ and $e=(e_1, ..., e_N)$ for the fiber coordinates of $E$. Then we may assume that $\sigma$ is the map
$$\sigma(y,e) = y_1 e_1 + \ldots + y_N e_N.$$

We can split $\sigma$ away from $0$ using the section $\check{\sigma} \colon \bR_V \to E|_V$ defined by
$$\check{\sigma} (y) = \begin{cases}
( y, \frac{y}{|y|} ) \text{ if } \frac{\epsilon}{2} < |y| < \epsilon \\
\\
( y, \frac{2}{\epsilon} y) \text{ if } |y| \leq \frac{\epsilon}{2}
\end{cases}$$
After smoothing and extending $\check{\sigma}$ to $Y$ (and possibly rescaling), we have a splitting and obtain a section $\xi$ of $E$ as the graph of $\check{\sigma}$. 

It is then immediate that $[C(s)] \cap [\xi] = \pm [0] \in H_0(Y(\sigma))$, where the sign is determined by the orientation of $E$ and hence
$$[Y]_{\loc,\sigma}\virt = \pm [0] \in H_0(Y(\sigma)).$$
\end{exam}

\subsection{Cosection localization for $(-2)$-shifted symplectic derived schemes} \label{cos loc for derived sch}

The methods developed in this section together with the strategy employed in the proof of Theorem~\ref{vanishing for non-degenerate cosection} allow us to significantly generalize its statement. Before we proceed, we introduce a bit of terminology.

\begin{defi}
A \emph{meromorphic cosection} on a derived scheme $\unxX$ is a cosection $\sigma$ defined on an open derived subscheme $\unuU (\sigma)$ of $\unxX$.

When $\unxX$ is $(-2)$-shifted symplectic with symplectic form $\omega_{\unxX}$, we say that $\sigma$ is non-degenerate if it is a non-degenerate cosection of the $(-2)$-shifted symplectic derived scheme $\left( \unuU(\sigma), \omega_{\unuU (\sigma)} \right)$, where $\omega_{\unuU (\sigma) } := \omega_{\unxX}|_{\unuU (\sigma) }$. If $U(\sigma)$ is the underlying topological space of $\unuU (\sigma)$, we use the notation $X(\sigma) = X - U (\sigma)$ for the locus where $\sigma$ is not defined or degenerate.
\end{defi}

We are now in position to generalize Theorem~\ref{vanishing for non-degenerate cosection}.

\begin{thm} \label{localization for derived scheme non-degenerate cos}
Let $(\unxX, \omega_{\unxX})$ be a proper, oriented $(-2)$-shifted symplectic derived scheme of virtual dimension $n$ with a meromorphic non-degenerate cosection $\sigma$. Then there exists a cosection localized virtual fundamental class $[\xX_{dm}]_{\loc, \sigma}\virt \in H_n(X(\sigma))$ satisfying
$$i_* [\xX_{dm}]_{\loc, \sigma}\virt = [\xX_{dm}]\virt \in H_n(X)$$
where $i \colon X(\sigma) \to X$ is the inclusion map.
\end{thm}

\begin{proof}
Since $X(\sigma)$ is a Zariski closed subset of $X$, it is an ENR and therefore there exists an open neighbourhood $V$ of $X(\sigma)$ that retracts onto $X(\sigma)$ such that the natural map
$$H_*(X(\sigma)) \lr H_*(V)$$
induced by inclusion is an isomorphism.

We may now follow the same strategy as in the proof of Theorem~\ref{vanishing for non-degenerate cosection} to associate to $(\unxX, \omega_{\unxX})$ a d-manifold $\xX_{dm}$ together with the data of an open sub-d-manifold $\wW$ whose underlying topological space $W$ is contained in $U(\sigma)$ and is an open neighbourhood of $X-V$ inside $U(\sigma)$ and a surjective real cosection $\sigma_\wW$ on $\wW$.

Theorem-Definition~\ref{localization by real cosection} and Remark~\ref{meromorphic cosection remark} then imply the existence of a localized virtual fundamental class $[\xX_{dm}]_{\loc, \sigma}\virt$ in $H_n(X-W)$ and hence in $H_n(V) \cong H_*(X(\sigma))$, which is independent of all choices involved and thus well defined and satisfies the required properties.
\end{proof}

\section{Localization by Complex Cosections} \label{Complex Cosections Section}

As usual, let $\xX$ be a compact, oriented d-manifold, $X_{\cC^\infty}$ the associated $\cC^\infty$-scheme and $X$ the underlying topological space. In this section, we study a certain kind of complex cosection that has similar properties to a cosection for an algebraic scheme with perfect obstruction theory. This allows us to localize in a more robust and systematic way of algebraic flavor compared to the previous section. 

Similarly to the case of real cosections, we write $\bC_X$ to denote both the trivial complex line bundle on $X_{\cC^\infty}$ and the trivial complex line bundle on $X$.

\begin{defi}
A \emph{complex cosection} is a morphism $\sigma \colon \Ob_{\xX} \to \bC_{X}$ of sheaves on $X_{\cC^\infty}$.
\end{defi}

The motivation behind the results of this section is the following observation: Suppose that $X$ is a complex projective scheme equipped with a perfect obstruction theory with cosection $\sigma$. Then, writing as usual $i \colon X(\sigma) \to X$ for the inclusion, there is a cosection localized virtual fundamental cycle $[X]\virt_{\loc, \sigma} \in H_*(X(\sigma))$ satisfying $i_* [X]\virt_{\loc, \sigma} = [X]\virt \in H_*(X)$, defined by an explicit formula in \cite{KiemLiCosection} which utilizes the blowup of $X$ along $X(\sigma)$. 

In \cite[Section~14.5]{JoyceDMan}, Joyce defines a truncation functor from schemes with perfect obstruction theory to d-manifolds, which acts functorially on their virtual fundamental classes. Applying this functor to $X$, one gets a d-manifold $\xX$ with an induced complex cosection $\sigma$ and a virtual fundamental class, which is thus automatically naturally localized by the cosection $\sigma$ by an explicit formula (see equation \eqref{formula of localized Gysin map} and Theorem-Definition~\ref{localization by complex cosection} below for the formula in our context).

One can then ask how close a d-manifold $\xX$ with a complex cosection $\sigma$ can be to a truncation of an algebraic (or complex analytic) scheme with a perfect obstruction theory with cosection to admit this more robustly described cosection localized virtual fundamental class.

As we will see, a sufficient set of holomorphicity assumptions regarding the topology of $\xX$ and the cosection $\sigma$ is given by the following setup. From now on in this section, we proceed to make these assumptions.

\begin{setup} \label{holomorphicity assumption} We say that $\xX$ and $\sigma$ come from geometry if:
\begin{enumerate}
    \item $X$ admits the structure of a complex analytic space $X^{an}$ inducing a structure of a $\cC^\infty$-scheme on $X$, denoted by $X^{an}_{\cC^\infty}$. 
    \item There exists a closed embedding $X^{an} \hookrightarrow W$, where $W$ is a complex manifold.
    \item There exists a closed embedding of $\cC^\infty$-schemes $X_{\cC^\infty}^{an} \hookrightarrow X_{\cC^\infty}$. 
    \item The locus $X(\sigma)$, where $\sigma$ is not surjective, is a complex analytic closed subset of $X^{an}$.
    \item The image of $\sigma|_{X_{\Coo}^{an}}$ is an $\oO_{X_{ \cC^\infty}^{an}}$-submodule of $\bC_X$ generated by an ideal $I \sub \oO_{X^{an}}$ whose zero locus is $X(\sigma)$.
\end{enumerate}
\end{setup}

We briefly clarify the above conditions. 

In (1), if $X^{an}$ is analytically locally described as the vanishing locus of holomorphic functions $f_1, ..., f_r$ defined on a polydisc $D$ in $\bC^N$, $X_{\cC^\infty}^{an}$ is locally described by the spectrum of the $\cC^\infty$-ring $\cC^\infty(D) / (\mathrm{Re} f_1, \mathrm{Im} f_1, ..., \mathrm{Re} f_r, \mathrm{Im} f_r)$. Clearly holomorphic transition functions between such charts induce smooth transition functions so that we obtain a $\cC^\infty$-scheme.

In (5), if $I$ is locally generated by a set of holomorphic functions $f_1, ..., f_r$ on $X^{an}$ whose common zero set is $X(\sigma)$, we consider the $\oO_{X_{\Coo}^{an}}$-submodule of $\bC_X = \bR_X \oplus \bR_X$ generated by the pairs 
$$(\mathrm{Re} f_1, \mathrm{Im} f_1), (-\mathrm{Im} f_1, \mathrm{Re} f_1), ..., (\mathrm{Re} f_r, \mathrm{Im} f_r), (-\mathrm{Im} f_r, \mathrm{Re} f_r)$$
of functions on $X_{\Coo}^{an}$. (5) requires the image of $\sigma|_{X_{\Coo}^{an}}$ to be locally of this form for some choice of generators of $I$.

Therefore, Setup~\ref{holomorphicity assumption} essentially says that $\xX$ and the cosection $\sigma$ are determined by complex analytic data up to possible thickening of the underlying $\cC^\infty$-scheme. 

\begin{rmk}
By design, if $X^{an}$ is a complex analytic space with an embedding $X^{an} \to W$ into a complex manifold and a perfect obstruction theory with holomorphic cosection $\sigma^{an}$, then the d-manifold $\xX$ obtained by applying the truncation functor in \cite[Section~14.5]{JoyceDMan}, and its induced complex cosection $\sigma$ satisfy the conditions of Setup~\ref{holomorphicity assumption}. The same holds for a complex scheme $X$ endowed with the same data.
\end{rmk}

\medskip

By Theorem-Definition~\ref{presentation as a d-manifold}, $\xX$ is equivalent to a principal d-manifold $\sS_{Y,E,s}$. We fix a choice of such presentation
\begin{equation} \label{principal model}
    \xX \simeq \sS_{Y,E,s}.
\end{equation}

Thus we have an embedding $X_{\cC^\infty} \to Y$, given by the data of an open $Y \sub \bR^N$, a smooth, oriented vector bundle $E$ on $Y$ of rank $r$, a smooth section $s$ of $E$ with $X$ being its zero locus and a surjection $p \colon E|_{X_{\Coo}} \to \Ob_{\xX}$. 

Write $\sigma'$ for the composition $\sigma \circ p \colon E|_{X_{\Coo}} \to \bC_X$ and $U(\sigma) = X - X(\sigma)$, the locus where $\sigma'$ is surjective. Let 
$$E(\sigma) = E|_{X(\sigma)} \cup \ker \left( \sigma' |_{U(\sigma)} \right). $$

Write $j \colon E|_{X(\sigma)} \to E(\sigma)$ for the inclusion.

In order to localize $[\xX]\virt$, we need to use an appropriate notion of resolution of the cosection, as in \cite{KiemLiCosection} and \cite{KiemLiQuantum}.

\begin{defi}
We say that a proper map $\rho \colon \ti{X} \to X$ of topological spaces is a \emph{$\sigma$-regularizing map} with respect to the presentation \eqref{principal model} if:
\begin{enumerate}
    \item $\rho|_{\rho^{-1}(U(\sigma))} \colon \rho^{-1}(U(\sigma)) \to U(\sigma)$ is a homeomorphism. Write $D$ for the complement of $U(\sigma)$ in $\ti{X}$.
    \item The pullback $\rho^* \sigma' \colon \rho^*E|_X \to \bC_{\ti{X}}$ factors through a surjection 
    \begin{align} \label{loc 4.2'}    
        \rho^*E|_X \to L
    \end{align} 
    where $L$ is a complex line bundle on $\ti{X}$ together with a continuous section $t \in \Gamma(\ti{X}, L^\vee)$ whose zero locus is $D$.
\end{enumerate}

We write $\rho(\sigma) \colon D \to X(\sigma)$ for the induced map. Moreover let $E'$ be the kernel bundle of the surjection \eqref{loc 4.2'}. We get an induced commutative diagram of topological spaces
\begin{align} \label{lifting diagram}
    \xymatrix{
    E' \ar[r] \ar[d]_-{\rho'} & \rho^* E|_X \ar[d] \\
    E(\sigma) \ar[r] & E|_X.
    }
\end{align}

We say that a $\sigma$-regularizing map satisfies \emph{homological lifting} if any homology class $\gamma \in H_i(E(\sigma))$ can be written as
\begin{align} \label{decomposition of gamma}
    \gamma = j_* \alpha + \rho'_* \beta
\end{align}
for homology classes $\alpha \in H_i(E|_{X(\sigma)})$ and $\beta \in H_i(E')$.
\end{defi}

The next proposition shows that when $\xX$ and $\sigma$ come from geometry, there exists an induced $\sigma$-regularizing map that satisfies homological lifting.

\begin{prop} \label{prop 4.10}
Suppose that the conditions in Setup~\ref{holomorphicity assumption} hold and we have fixed a presentation \eqref{principal model}. Then there exists an induced $\sigma$-regularizing map $\rho \colon \ti{X} \to X$ that satisfies homological lifting.
\end{prop}

\begin{proof}
By parts (1) and (2) of Setup~\ref{holomorphicity assumption}, we have a closed embedding $X^{an} \to W$. In particular, at the level of $\cC^\infty$-schemes we get an embedding
\begin{align} \label{loc 4.3}
    X_{\cC^\infty}^{an} \lr W_{\cC^\infty}.
\end{align}

By definition, pulling back $E$ from $Y$ via the embedding $X_{\cC^\infty} \to Y$ gives a vector bundle $E|_{X_{\cC^\infty}}$ on $X_{\cC^\infty}$. 

By part (3) of Setup~\ref{holomorphicity assumption}, we can further pull it back to a vector bundle $E|_{X_{\cC^\infty}^{an}}$ on $X_{\cC^\infty}^{an}$. Pulling back $\sigma'$ gives a cosection
\begin{align} \label{loc 4.4}
    E|_{X_{\cC^\infty}^{an}} \lr \bC_{X_{\cC^\infty}^{an}}.
\end{align}

At the topological level, both vector bundles $E|_{X_{\cC^\infty}}$ and $E|_{X_{\cC^\infty}^{an}}$ give the vector bundle $E|_X$ on the underlying topological space $X$.

The existence of the embedding \eqref{loc 4.3} and the bundle $E|_{X_{\cC^\infty}^{an}}$ implies (for example, by the argument in the proof of \cite[Theorem~4.34]{JoyceDMan}) that there exists an open neighbourhood $V$ of $X$ in $W$ such that $E|_X$ extends to a smooth vector bundle $E_V$ on $V$. Moreover, using parts (4) and (5) of Setup~\ref{holomorphicity assumption} and the cosection \eqref{loc 4.4}, up to shrinking $V$, we may extend the cosection $\sigma' \colon E|_X \to \bC_X$ to a smooth morphism of vector bundles
\begin{align*}
    \sigma_V \colon E_V \lr \bC_V.
\end{align*}

Let $\tau \colon \ti{V} \to V$ be the blowup of $V$ along $X(\sigma)$ with exceptional divisor $D_V$ and set $\ti{X} = \ti{V} \times_V X$ with induced proper morphism $\rho \colon \ti{X} \to X$. 

For two such choices $V_1, V_2$ and blowup maps $\tau_1 \colon \ti{V}_1 \to V_1$ and $\tau_2 \colon \ti{V}_22 \to V_2$, let their intersection inside $W$ be $V_{12} = V_1 \cap V_2$ and $\tau_{12} \colon \ti{V_{12}} \to V_{12}$ be the blowup of $V_{12}$ along $X(\sigma)$. For $i=1,2$, since $X(\sigma) \sub V_{12}$ is a closed embedding and $V_{12}$ is open in both $V_1$ and $V_2$, we  have diagrams with Cartesian squares
\begin{align*}
\xymatrix{
\ti{X} \ar[r] \ar[d]^-{\rho} & \ti{V}_{12} \ar[r] \ar[d]^-{\tau_{12}} & \ti{V}_i \ar[d]^-{\tau_i} \\
X \ar[r] & V_{12} \ar[r] & V_i
}
\end{align*}
and therefore $\rho$ is independent from the choice of neighbourhood $V$. 

We now have a complex line bundle $L = \oO_{\ti{V}}(-D_V)|_{\ti{X}}$ and a morphism 
$$\tau^* \sigma_V \colon \tau^* E_V \to \oO_{\ti{V}}(-D_V)$$
whose restriction to $\ti{X}$ gives a surjection
\begin{align*}
    \rho^* \sigma' \colon \rho^* E|_X \lr L.
\end{align*}
We thus see that $\rho \colon \ti{X} \to X$ is $\sigma$-regularizing. 

By construction, the diagram~\eqref{lifting diagram} gives rise to a commutative diagram
\begin{align} \label{loc 4.10'}
    \xymatrix{
    E' \ar[r] \ar[d]_{\rho'} & \tau^* E_V \ar[d] \\
    E(\sigma) \ar[r] & E_V.
    }
\end{align}
Then the proof of \cite[Lemma~2.3]{KiemLiQuantum} applies verbatim to this setting and implies that $\rho$ satisfies homological lifting.
\end{proof}

\begin{rmk} \label{Remark 5.6}
If we have two closed embeddings $X^{an} \to W_1$ and $X^{an} \to W_2$ of $X^{an}$ into complex manifolds $W_1$ and $W_2$, a standard argument using the diagonal embedding $X^{an} \to W_1 \times W_2$ and the projections $W_1 \times W_2 \to W_i$ for $i=1,2$ shows that the map $\rho$ is independent from the particular choice of closed embedding $X^{an} \to W$ and thus determined canonically by the rest of the data of Setup~\ref{holomorphicity assumption}. 

What is essential is that such an embedding exists. This is true automatically for example when $X^{an}$ is quasi-projective.
\end{rmk}

Given a $\sigma$-regularizing map $\rho \colon \ti{X} \to X$ satisfying homological lifting, by \cite{KiemLiQuantum} we have a localized Gysin map
\begin{align} \label{localized Gysin map}
    0_{E|_{X}, \sigma}^! \colon H_*(E(\sigma)) \lr H_{*-r}(X(\sigma))
\end{align}
given by the formula
\begin{align} \label{formula of localized Gysin map}
    0_{E|_X, \sigma}^! \gamma = 0_{E|_{X(\sigma)}}^! \alpha - \rho(\sigma)_* \left( e(L^\vee, t) \cap 0_{E'}^! \beta \right)
\end{align}
where $\gamma \in H_i(E(\sigma))$ and we have written $\gamma = j_* \alpha + \rho'_* \beta$ for some $\alpha \in H_i(E|_{X(\sigma)})$ and $\beta \in H_i(E')$.

By the proof of \cite[Theorem~3.2]{KiemLiQuantum}, the Gysin map $0_{E|_X, \sigma}^!$ is independent from the particular choice of $\alpha$ and $\beta$.

If $i \colon X(\sigma) \to X$ denotes the inclusion, it is shown in \cite[Corollary~2.9]{KiemLiCosection} that it satisfies
\begin{align*}
    0_{E|_X}^! \circ j_* = i_* \circ 0_{E|_{X}, \sigma}^! \colon H_*(E|_X) \to H_{*-r}(X).
\end{align*}

Now, by Proposition~\ref{support of cone}, $C(s)$ is contained in $E(\sigma)$ and therefore by the definition \eqref{definition of homological normal cone} it follows that we obtain a class
\begin{align*}
    [C(s)] \in H_N(E(\sigma)).
\end{align*}

We are now in position to explicitly define the cosection localized virtual fundamental class of $\xX$.

\begin{thm-defi} \label{localization by complex cosection}
Let $\xX$ be a compact, oriented d-manifold of virtual dimension $n$ and $\sigma \colon \Ob_{\xX} \to \bC_X$ a complex cosection, satisfying the conditions of Setup~\ref{holomorphicity assumption}. Then the cosection localized virtual fundamental class of $\xX$ is defined by the formula
\begin{align*}
    [\xX]_{\loc,\sigma}\virt := 0_{E|_{X}, \sigma}^! [C(s)] \in H_{n}(X(\sigma)).
\end{align*}
It satisfies
\begin{align*}
    i_* [\xX]_{\loc,\sigma}\virt = [\xX]\virt \in H_n(X).
\end{align*}
\end{thm-defi}

Finally, we have not addressed the potential dependence of $[\xX]_{\loc,\sigma}\virt$ on the presentation \eqref{principal model} of $\xX$ as a principal d-manifold, i.e. on the data $Y$, $E$, $s$.

\begin{prop}
$[\xX]_{\loc,\sigma}\virt$ is independent from the choice of presentation $\xX \simeq \sS_{Y,E,s}$ and thus well defined.
\end{prop}

\begin{proof}
Firstly, the $\sigma$-regularizing map $\rho \colon \ti{X} \to X$ by construction can only a priori depend on the choice of closed embedding $X \hookrightarrow W$ and of an open neighbourhood $V$ of $X$ inside $W$. By the proof of Proposition~\ref{prop 4.10} and Remark~\ref{Remark 5.6}, it is independent of these choices and thus independent from the particular choice of presentation.

Suppose now that $\xX \simeq \sS_{Y_1,E_1,s_1} \simeq \sS_{Y_2,E_2,s_2}$. As in the proof of Proposition~\ref{prop 5.2}, up to possibly shrinking $Y_1$ around $X$, we can write this equivalence in the following standard form: we have a submersion $f \colon Y_1 \to Y_2$ and a morphism of vector bundles $\hat{f} \colon E_1 \to f^* E_2$ such that $\hat{f} \circ s_1 = f^* s_2$. The morphism of tangent complexes
\begin{align*}
    \xymatrix{
    T_{Y_1} |_X \ar[r]^-{\dd s_1} \ar[d]_-{\dd f} & E_1|_X \ar[d]^-{\hat{f}} \\
    f^* T_{Y_2} |_X \ar[r]^-{\dd (f^* s_2)} & f^\ast E_2 |_X
    }
\end{align*}
is a quasi-isomorphism on fibers over $x \in X$ and thus we must have that $E_1|_X \xrightarrow{\hat{f}} f^* E_2|_X$ is surjective, so after further possible shrinking of $Y_1$, we may assume that $\hat{f}$ is surjective. 

Using the notation of diagram \eqref{lifting diagram}, we have an induced diagram with Cartesian squares
\begin{align*}
    \xymatrix{
    E_1' \ar[d]_-{\rho_1'} \ar[r]^-{\hat{f}'} & (f^\ast E_2)' \ar[d] \ar[r]^-{f'} & E_2' \ar[d]^-{\rho_2'}\\
    E_1(\sigma) \ar[r]^-{\hat{f}(\sigma)} & (f^* E_2)(\sigma) \ar[r]^-{f(\sigma)} & E_2(\sigma) \\
    E_1|_{X(\sigma)} \ar[u]^-{i_1} \ar[r]^-{\hat{f}|_{X(\sigma)}} & (f^* E_2)|_{X(\sigma)} \ar[u] \ar[r]^-{f|_{X(\sigma)}} & E_2|_{X(\sigma)} \ar[u]_-{i_2}
    }
\end{align*}
where the horizontal arrows are smooth and the vertical arrows proper.

Write $\hat{g}' = f' \circ \hat{f}'$, $g(\sigma) = f(\sigma) \circ \hat{f}(\sigma)$ and $g|_{X(\sigma)} = f|_{X(\sigma)} \circ \hat{f}|_{X(\sigma)}$ for the (submersive) compositions of the horizontal arrows.

Using base change for Cartesian squares with vertical arrows that are submersions and proper horizontal arrows (cf. for example \cite[Equations~(2.12), (2.23)]{KiemLiQuantum}), the local description of the equivalence $\sS_{Y_1, E_1, s_1} \simeq \sS_{Y_2, E_2, s_2}$ given in the proof of Proposition~\ref{prop 5.2} implies that in the notation of \eqref{loc 3.2} we have
$$(s_{1,\ell})_* [Y_1 \times (\bR^\ell \backslash \lbrace 0 \rbrace] = \hat{g}^\ast (s_{2, \ell})_* [Y_2 \times (\bR^\ell \backslash \lbrace 0 \rbrace]$$ 
where $\hat{g}$ is the composition $E_1 \xrightarrow{\hat{f}} f^* E_2 \to E_2$.

It thus follows from the definition of the cones $[C(s_1)]$ and $[C(s_2)]$ and Proposition~\ref{support of cone} that
$$[C(s_1)] = \hat{f}(\sigma)^\ast [C(f^\ast s_2)] = \hat{f}(\sigma)^\ast f(\sigma)^\ast [C(s_2)] = g(\sigma)^\ast [C(s_2)] $$ 
and therefore if 
$$[C(s_2)] = (i_2)_\ast \alpha_2 + (\rho_2')_\ast \beta_2$$ in the notation of \eqref{decomposition of gamma}, we will have
$$[C(s_1)] = (i_1)_\ast \alpha_1 + (\rho_1')_\ast  \beta_1$$
where
$$\alpha_1 = (g|_{X(\sigma)})^\ast \alpha_2, \ \beta_1 = (\hat{g}')^* \beta_2. $$

By the usual functoriality properties of Gysin maps (cf. for example \cite[Section~2.4]{KiemLiQuantum}), we obtain
\begin{align*}
    0_{E_1|_{X(\sigma)}}^! \alpha_1 =  0_{E_2|_{X(\sigma)}}^! \alpha_2, \ 0_{E_1'}^! \beta_1 = 0_{E_2'}^! \beta_2
\end{align*}
so the formula of the localized Gysin map \eqref{formula of localized Gysin map} immediately implies that
$$ 0_{E_1|_X, \sigma}^! [C(s_1)] = 0_{E_2|_X, \sigma}^! [C(s_2)] $$
as we want.
\end{proof}

\begin{rmk}
In general, given a cosection $\sigma$ on $\xX$, it is not clear if it is possible to endow the locus $X(\sigma)$ with a natural choice of structure of a d-manifold.

This is intuitive from the point of view of complex analytic spaces with perfect obstruction theory, as a cosection does not necessarily induce a perfect obstruction theory on its vanishing locus.
\end{rmk}

\section{Application: Vanishing of Stable Pair Invariants of Hyperk\"{a}hler Fourfolds} \label{Application Section}

Let $W$ be a smooth, projective Calabi-Yau fourfold. In \cite{CaoMaulikToda}, the authors study the stable pair invariants of $W$ and claim that they vanish when $W$ is hyperk\"{a}hler, anticipating the existence of a cosection localization theory for d-manifolds. In this section, we prove their claim using the theory developed in the previous parts of the paper.\medskip

A stable pair on $W$ is the data of
\begin{equation*}
    (F, s), \ F \in \Coh W,\ s \colon \oO_W \lr F
\end{equation*}
where $F$ is a pure one-dimensional sheaf on $W$ and $s$ is surjective in dimension one. 

Fix $\beta \in H_2(W, \bZ)$ and $n \in \bZ$. Then the moduli space $P_n(W,\beta)$ is a projective scheme which parameterizes stable pairs viewed as two-term complexes 
\begin{align*}
    I = [ \oO_W \xrightarrow{s} F] \in D^b(\Coh W)
\end{align*}
satisfying $[F]=\beta$ and $\chi(F) = n$.

By \cite[Lemma~1.3]{CaoMaulikToda}, $X = P_n(W, \beta)$ is the truncation of a $(-2)$-shifted symplectic derived scheme $\unxX$, which admits a choice of orientation. Writing $\bI = [\oO_{W \times X} \to \bF]$ for the universal stable pair and $\pi_X \colon X \times W \to X, \pi_W \colon X \times W \to W$ for the projection maps, we have
\begin{align*}
    \bT_{\unxX}|_X = \RHOM_{\pi_X}(\bI,\bI)_0[1]
\end{align*}
so that at a point $x=[I] \in X$ we have $\Ob(\unxX, x) = \Ext^2(I,I)_0$ and the quadratic form $q_x = q_I$ induced by the symplectic form coincides with the Serre duality pairing.\medskip

Suppose now that $W$ is hyperk\"{a}hler with holomorphic symplectic form $\eta$. Then, using the Atiyah class of $\bI$ and the symplectic form, it is shown in \cite{CaoMaulikToda} that $\unxX$ admits a strong cosection, whose associated weak cosection is non-degenerate.

We briefly recall the construction for the convenience of the reader. 

Let 
$$\At(\bI) \in \Ext^1(\bI, \bI \otimes \pi_W^* \Omega_W )$$
be the relative Atiyah class of $\bI$. We then get morphisms
\begin{align*}
    \wedge \frac{\At(\bI)^2}{2} \colon & \bT_{\unxX}|_X[1] = \RHOM_{\pi_X}(\bI,\bI)_0[2] \lr \RHOM_{\pi_X} \left( \bI,\bI \otimes \pi_W \Omega_W^2 [2] \right)[2] \\
    \lrcorner \ \eta \colon & \RHOM_{\pi_X}\left( \bI,\bI \otimes \pi_W \Omega_W^2 [2] \right)[2] \lr \RHOM_{\pi_X}(\bI,\bI[2])[2] \\
    \tr \colon & \RHOM_{\pi_X}(\bI,\bI[2])[2] \lr \RHOM_{\pi_X}(\oO_{X\times W},\oO_{X\times W}[2])[2] = \oO_X \oplus \oO_X[-4]
\end{align*}
Projecting onto $\oO_X$, their composition gives a cosection
\begin{align*}
    \sigma \colon \bT_{\unxX}|_X [1] \lr \oO_X
\end{align*}
which at the point $x = [I] \in X$ corresponds to a morphism
\begin{align} \label{sigma I}
    \sigma_x = \sigma_I \colon \Ob(\unxX,x) = \Ext^2(I,I)_0 \lr H^4(W, \oO_W).
\end{align}

We have the following proposition.

\begin{prop} \emph{\cite[Proposition~2.9]{CaoMaulikToda2}} \emph{\cite[Proposition~2.18]{CaoMaulikToda}} 
Let $W$ be a smooth, projective hyperk\"{a}hler fourfold and $X$ a connected scheme. Let $\bI$ be any perfect complex on $X \times W$ and write $\bI |_x = I$ for $x \in X$. Then if $\ch_3(I) \neq 0$ or $\ch_4(I) \neq 0$, $\sigma_I$ is surjective and:
\begin{enumerate}
    \item If $\ch_4(I) \neq 0$, $\sigma_I$ admits a $q_I$-orthogonal splitting and $q_I$ is non-degenerate on $H^4(W, \oO_W)$.
    \item If $\ch_4(I)=0$ and $\ch_3(I) \neq 0$, let $\kappa_W \in H^1(W, T_W)$ be the Kodaira-Spencer class dual to $\ch_3(I)$. Similarly to the above, we obtain a second strong cosection
    \begin{align*}
        \tau \colon \pi_W^* \kappa_W \circ \left( \wedge \At(\bI) \right) \colon \bT_{\unxX}|_X[1] \lr \oO_X.
    \end{align*}
    The map $\sigma_I \oplus \tau_I$ is surjective  and admits a $q_I$-orthogonal splitting such that $q_I$ is non-degenerate on $H^4(W, \oO_W) \oplus H^4(W, \oO_W)$ with matrix $\begin{pmatrix} 0 & 2 \\ 2 & 0 \end{pmatrix}$. In particular, 
    \begin{equation}  \label{sigma I and tau I}
    \sigma_I + \tau_I \colon \Ext^2(I,I)_0 \lr H^4(W, \oO_W)
    \end{equation}
    admits a $q_I$-orthogonal splitting and $q_I$ is non-degenerate on $H^4(W, \oO_W)$.
\end{enumerate}
\end{prop}

Using the proposition, it is immediate by Definition~\ref{cosection for derived schemes} that if $n \neq 0$, the cosection $\sigma$ for $P_n(W, \beta)$ is non-degenerate, and if $\beta \neq 0$, the cosection $\sigma + \tau$ given pointwise by \eqref{sigma I and tau I} is non-degenerate. 

Therefore by Theorem~\ref{vanishing for non-degenerate cosection} we immediately get the following vanishing result for stable pair invariants on hyperk\"{a}hler fourfolds, which is precisely \cite[Claim~2.19]{CaoMaulikToda}.

\begin{thm} \label{vanishing of PT4}
Let $W$ be a smooth, projective hyperk\"{a}hler fourfold and $P_n(W, \beta)$ the moduli space of stable pairs with $n \neq 0$ or $\beta \neq 0$. Then $[P_n(W, \beta)]\virt = 0$.
\end{thm}

\bibliography{Master}
\bibliographystyle{alpha}

\end{document}